\documentclass[letterpaper,11pt]{amsart}

\usepackage[margin=2cm]{geometry}
\usepackage{amscd, amssymb}
\usepackage{amsmath,amscd}
\usepackage[driverfallback=hypertex]{hyperref}
\usepackage{graphics}
\usepackage{graphicx}
\usepackage{color}
\usepackage{bm}
\usepackage{enumitem}
\setlist[enumerate,1]{label={(\alph*)}}
\setlist[enumerate,2]{label={(\roman*)}}
\input xy
\xyoption{all}

\newtheorem{theorem}{Theorem}[section]
\newtheorem{prop}[theorem]{Proposition}

\newtheorem{lem}[theorem]{Lemma}

\theoremstyle{definition}
\newtheorem{dfn}[theorem]{Definition}
\newtheorem{nn}[theorem]{Notation}

\theoremstyle{remark}

\newtheorem{obs}[theorem]{Observation}

\newcommand{\ignore}[1]{}

\newcommand{\CC}{{{c}}}

\usepackage{color}

\definecolor{Red}{rgb}{1,0,0}
\definecolor{Blue}{rgb}{0,0,1}
\definecolor{Olive}{rgb}{0.41,0.55,0.13}
\definecolor{Yarok}{rgb}{0,0.5,0}
\definecolor{Green}{rgb}{0,1,0}
\definecolor{MGreen}{rgb}{0,0.8,0}
\definecolor{DGreen}{rgb}{0,0.55,0}
\definecolor{Yellow}{rgb}{1,1,0}
\definecolor{Cyan}{rgb}{0,1,1}
\definecolor{Magenta}{rgb}{1,0,1}
\definecolor{Orange}{rgb}{1,.5,0}
\definecolor{Violet}{rgb}{.5,0,.5}
\definecolor{Purple}{rgb}{.75,0,.25}
\definecolor{Brown}{rgb}{.75,.5,.25}
\definecolor{Grey}{rgb}{.5,.5,.5}

\begin{document}
\title{A sharp threshold for spanning $2-$spheres in random $2-$complexes}
\author{Zur Luria and Ran J. Tessler}
\address{Institute for Theoretical Studies, ETH Z\"urich}

\begin{abstract}
A Hamiltonian cycle in a graph is a spanning subgraph that is homeomorphic to a circle. With this in mind, it is natural to define a Hamiltonian $d$-sphere in a $d$-dimensional simplicial complex as a spanning subcomplex that is homeomorphic to a $d$-dimensional sphere.

We consider the Linial-Meshulam model for random simplicial complexes, and prove that there is a sharp threshold at $p=\sqrt{\frac{e}{\gamma n}}$ for the appearance of a Hamiltonian $2$-sphere in a random $2$-complex, where $\gamma = 4^4/3^3$.
\end{abstract}

\maketitle

\section{Introduction}
A classical theorem of P\'{o}sa \cite{Posa} states that the threshold for the appearance of a Hamiltonian cycle in the random graph $G(n,p)$ is $\frac{\log(n)}{n}$. At first sight, this result is surprising, as a first moment estimate indicates a threshold of $\frac{1}{n}$. A second thought shows that in fact, below $\Theta(\frac{log n}{n})$ there are isolated vertices with high probability, and hence no Hamiltonian cycles.

We are interested in the analogous question in higher dimensions. For this purpose, one must define a generalization of the concept of a Hamiltonian cycle to simplicial complexes, and indeed, several such definitions exist in the literature (see for example \cite{Be76, KK99}). The most popular defines a Hamiltonian cycle in a $d$-dimensional complex $X$ to be an ordering of the vertices such that every $(d+1)$ consecutive vertices form a simplex of $X$.
In this definition, however, a Hamiltonian cycle remains a ``1-dimensional" object.

Another way to view a Hamiltonian cycle is as a spanning subcomplex that is homeomorphic to a circle. From this point of view, the following definition is natural.

\begin{dfn}
A Hamiltonian $d$-sphere in a $d$-dimensional simplicial complex is a spanning subcomplex that is homeomorphic to a $d$-dimensional sphere.
\end{dfn}

Many questions suggest themselves. Are there sufficient conditions for a simplicial complex to contain a Hamiltonian sphere? Another question is: Under what conditions is there an efficient algorithm to find a Hamiltonian sphere in a given complex?

In this paper, we investigate the appearance of a Hamiltonian sphere in the random setting.
Let $X_d(n,p)$ denote the Linial-Meshulam model for random $d$-complexes, which is an extension of the Erd\'{o}s-Renyi random graph model to simplicial complexes. A simplicial complex $X \sim X_d(n,p)$ has $n$ vertices, and each $d$-dimensional simplex is present with probability $p$, independently of the other simplices.
We ask the following question for $d=2$: What is the threshold for a random $X \sim X_d(n,p)$ to contain a Hamiltonian $d$-sphere?

Using results of Tutte \cite{Tutte}, a first moment argument reveals that the threshold should be at least $\frac{c}{\sqrt n}$ for a specific constant $c$.
In \cite{ParPer} it was proven that if $\{s_n\}_{n \geq 3}$ is an infinite family of Hamiltonian $2$-spheres, and if the degree of each vertex is uniformly bounded in $n$ by some fixed $\Delta$, then for $p=\omega(\frac{\Delta^2}{\sqrt{n}})$, $s_n$ will asymptotically almost surely appear as a spanning subcomplex of the random complex $X \sim X_2(n,p)$.

Let $X \sim X_2(n,p)$. For a set $S$ of triangles, we write $S \subseteq X$ to denote the event that the set of triangles in $X$ contains $S$.
Let $S_n$ be the set of triangulations of a sphere using $n$ labeled vertices, where double edges and loops are not allowed. Equivalently, $S_n$ is the set of $n$-vertex $2$-dimensional simplicial complexes that are homeomorphic to a sphere.

\begin{theorem}
\label{thm:main}
Let $\varepsilon>0$ be an arbitrarily small constant, and let $\gamma = 4^4/3^3$.
\begin{enumerate}
\item\label{it:first}
 If $p<(1-\varepsilon)\left(\frac{e}{\gamma n}\right)^{\frac{1}{2}}$ then the probability that $X \sim X_2(n,p)$ contains an element of $S_n$ is $o(1)$.
\item\label{it:second}
 If $p>(1+\varepsilon)\left(\frac{e}{\gamma n}\right)^{\frac{1}{2}}$ then the probability that $X \sim X_2(n,p)$ contains an element of $S_n$ is $1 - o(1)$.
\end{enumerate}
\end{theorem}

\subsection{The plan of the paper and the proof}
\subsubsection{The main ideas}
Since the proof has several stages, in this subsection we explain the idea behind the proof and why we need each step.

Observe that two random triangulations of the sphere are not likely to intersect, as the number of possible triangles over $n$ vertices is $\binom{n}{3},$ but the number of triangles in a triangulation is $O(n)$, and each triangle is contained by the same number of triangulations. Note that this is different from the one dimensional case (of hamiltonian cycles), and it hints that possibly, unlike that case, in $2$ dimensions the first moment prediction for threshold is true.

The first moment prediction for the threshold follows from Tutte's formula for the number of triangulations, Theorem \ref{thm:Tutte}.
This immediately gives a lower bound for the threshold. We prove that this is also an upper bound by applying the second moment technique.

By associating an indicator random variable for each triangulation and performing standard manipulations, we reduce the required estimate of the second moment method, to the question of estimating
\[\frac{1}{|S_n|^2}\sum_F p^{-|F|} \cdot |\{s,s'\in S_n,~F\subseteq s\cap s'\}|,\]
where $S_n$ is the set of $n$-vertex triangulations of the sphere, and $F$ is a \emph{completable} set of triangles, that is, a set of triangles which may be completed to a triangulation of the sphere.

The next step is to group completable sets $F$ together according to a specific set of parameters that determines their contribution to the sum. In Subsection \ref{subsec1} we define these parameters precisely. Roughly speaking, it includes the boundary of $F$, a planar, but not yet embedded, graph, together with a $2$-coloring of its faces, in which the white faces are those areas of the boundary graph that are covered by triangles in $F$. We also sum over the allocation of interior points in the white faces. The remaining black faces, which depend on the embedding, correspond to the areas of the sphere that are not covered by $F$. We call this new object a completable filled planar graph, or \emph{CFPG}.
In terms of summation over CFPGs, we want to evaluate
\begin{equation}\label{eq:1sketch}
\frac{1}{|S_n|^2} \sum_{H \in \tilde{P}_n, \eta_1,\eta_2 \in \text{Em}(H)}{p^{-T_H} \cdot |M(H)| \cdot N(H,\eta_1) \cdot N(H,\eta_2) }\end{equation}
where $T_H$ is the number of triangles in a completable $F$ whose boundary is $H$, $|M(H)|$ is the number of ways to triangulate the white faces (knowing the point allocation, which is part of the data), $\text{Em}(H)$ is the set of embeddings and $N(H,\eta_i)$ is the number of ways to triangulate the black faces, given the embedding $\eta_i,$ on which we also sum.

In order to explain better the idea of proof, we now simplify our problem by forgetting the summation over embeddings, and assuming implicitly that $H$ is embedded. We shall return to this point later in the sketch.

Let us consider the summation over CFPGs on $m$ vertices. Dividing by $n!$ takes us from summing over labeled graphs to summing over unlabeled graphs. We get an expression of the following form (still omitting the embeddings):
\begin{equation}\label{eq:2sketch}\frac{e^{O(m)}}{|S_n|^2}\frac{n!}{(n-m)!}\sum\left\{\frac{1}{|Aut(G)|}\sum_{k=0}^{n-m}\binom{n-m}{k} p^{-T_{G,k}} M(\bar{G},k) M(G,n-m-k)^2\right\},\end{equation}
where we now sum over graphs which are defined similarly to CFPGs, except that the allocation of the interior points is not determined. We also sum over $k,$ the total number of interior points that will be allocated to the white region. $\bar{G}$ is the same embedded graph, in which the color of the black faces is changed to white and vice versa.

Since there is an exponential (in $m$) number of embedded planar graphs on $m$ vertices, we may replace the sum by maximum, and pay another multiplicative factor of $e^{O(m)}.$ Let us also, for the moment, ignore the term $Aut(G),$ although it will also play a crucial role later.

A key observation for bounding the number of embeddings of a single connected component of $H$ is Lemma \ref{lem:estimation_of_triangulations_of_sphere_with_holes} which estimates the number of triangulations in a topological disk with several holes.

Plugging this estimate into the summand that corresponds to $H$ in \eqref{eq:1sketch} leads to the hope that the contribution of $H$ to \eqref{eq:1sketch} behaves like $c^m n^{-am-b},$ or more precisely, $c^m n^{-m/6},$ where $m$ is the number of vertices of $H.$
This estimate is indeed true for \emph{simple $H$}, but it turns out that complicated graphs, with many connected components and large $m$, may+ behave worse.

Lemma \ref{lem:estimation_of_triangulations_of_sphere_with_holes} hints that a graph $G$ maximizing the summand in \eqref{eq:2sketch} will have a single white and a single black face that touches more than $3$ connected components. Indeed, we prove in Lemma \ref{lem:effect of flattening} that given a graph $G$ ,we may find another graph, $G'$, obtained from $G$ by \emph{flattening}, which is a process of moving holes from the interior of one face to another. The new graph $G'$ will have a single face of each color with many holes, and its corresponding summand in \eqref{eq:2sketch} is greater than the summand corresponding to $G.$
After performing some analysis we bound the contribution of any flat graph $G$ by an expression of the form $(m/n)^{a(G)}m^{b(G)}.$
Moreover, $a,b$ may be calculated in an iterative way, by adding one face after the other. It turns out that the contribution of most faces increases $a$ and decreases $b$ by and amount that is proportional to the size of the added face. However, this is not the case for small faces, such as triangles. This is not a problem in the analysis, but in the bounds themselves.

Now the automorphism group comes to play - when we have many small faces the automorphism group tends to be big. This qualitative statement may be quantized. However, estimating the automorphism group is not an easy task. Instead we give a lower bound by counting how many small faces (of several types, with no more than $6$ vertices) we have, which have no other face bounded by them. With this information we may lower bound the automorphism group, and show that whenever the $a(G),-b(G)$ are not large enough, the autumorphism factor of in the summand of \eqref{eq:2sketch} compensates. Note that the fact we account for the numbers of small faces requires some caution in the flattening process (we do not want to take out the last face surrounded by a small face).

We so far described a proof under the simplifying assumption of neglecting the different embeddings of a given graph. However, this number may be huge.
Intuitively, given a connected component of a planar graph, a vertex whose removal disconnects the graphs to many pieces contributes many embeddings (for example - choose the cyclic order of these faces around the vertex, and get a factorial term), and also pairs of vertices whose removal disconnects the component into many pieces may be responsible for many embeddings (again, by ordering the pieces).
It turns out that the second type, the pairs, is worse for us. More specifically, since we saw above that large components affect the bound for the summand of \eqref{eq:2sketch} very well, it is seen that pairs of vertices $x,y$ with many \emph{bananas} which are quadrangular faces having $x,y$ as opposite vertices and the other two vertices are of degree $2$ give the worse contribution (triangles are also bad, from this point of view, but in our definition of triangulation at most one triangle of each color may lie on each edge).

The way to attack this point is to show how to reduce a graph with many bananas to a graph with a few bananas (this is more or less what we call a \emph{good graph}), and to estimate the difference in contributions along this process. We show that the good graphs are responsible to the dominant part of the contribution to \eqref{eq:2sketch}.

We then estimate how many embeddings good graphs have. We do it by proving a quantitative estimate for the classical theorem of Whitney which claims that a $3-$connected planar graph has a single embedding, up to an orientation inversion (see appendix).

Armed with these estimates, we return to the above scheme (with some care regarding the remaining bananas) and bound \eqref{eq:2sketch} by $\sum _{m=3}^n e^{O(m)}(m/n)^{-m/6}m^{-m/5000}.$ This was the fabula, we now move to the syuzhet.

\subsubsection{The actual structure of the proof}
We start Section \ref{sec:proof} by standard first and second moment arguments, this bounds the threshold from below and gives an estimate for the sum of some functional over special sets of triangles (\emph{completable}) we need to prove, in order to show that the bound for the threshold is indeed the threshold.
Subsection \ref{subsec1} shows that it is enough to estimate over \emph{good graphs}, which are roughly speaking graphs without too many faces which are rhombi (\emph{bananas}) with exactly two vertices, which are opposite, of degree greater than $2.$
In Subsection \ref{subsec2} we give an estimate for the number of embeddings of a graph based on its connectivity properties. Subsection \ref{subsec3} simplifies the sum of the functional which appeared in the second moment argument, replaces the functional by a simpler one, the sum by maximum, and the planar graphs by embedded planar graphs. Then, in Subsection \ref{subsec4} we show that up to irrelevant terms, the maximum of the new functional is obtained from embedded graphs which we call \emph{flat}, whose nesting picture is relatively simple.
We then analyze the functional for such graphs in \ref{subsec5}. The functional of a given graph on $m$ vertices (out of the $n$ vertices we start with) is, up to irrelevant factors, of the form $n^am^b,$ where $a,b$ depend on the graph. By analyzing the contribution of each component of the flat graph, and taking into account automorphisms, we estimate $a,b$ in terms of $m$ and finish the proof (Subsection \ref{subsec6}). In the appendix several technical and non technical lemmas are proven.

\subsection{Acknowledgements}
The authors would like to thank Kirushiga Kanthan and Amitos Solomon for discussions related to this work.

Z.L. and R.T. are supported by Dr. Max R\"ossler, the Walter Haefner Foundation and the ETH Z\"urich
Foundation.

\section{Planar Graphs}\label{sec:planar}

A planar graph is a graph that can be embedded in the plane. An embedded planar graph is a planar graph together with such an embedding, which is  defined up to an isotopy, that is, a continuous transformation that does not create intersections.



We will be interested in graphs that are embedded in the sphere. Such graphs can be thought of as planar graphs, if one chooses a face of the graph to be the outer face, and by stretching it maps the rest of the sphere onto a plane. Different choices for the outer face may however result in different embedded planar graphs. Since a planar graph on $m$ vertices has $O(m)$ faces, the number of embedded $m$-vertex planar graphs and the number of $m$-vertex graphs that are embedded in the sphere differ by at most a multiplicative factor of $O(m)$.

A classical result of Tutte \cite{Tutte} implies that the number of embedded planar graphs on $m$ unlabelled vertices is exponential in $m$.

\begin{lem}\label{lem:num_of_planar_graphs}
The number of embedded planar graphs on $m$ unlabelled vertices is at most
$c_1^m$, for a constant $c_1$.
\end{lem}

Indeed, Tutte gave a formula for the number $t(m)$ of $m$-vertex triangulations of the plane, implying that $t(m)\leq \gamma^{(1+o(1))m}$. An $m$-vertex triangulation has exactly $3m-6$ edges, and as every embedded $m$-vertex planar graph is a subgraph of an $m$-vertex triangulation of the plane, their number is at most $2^{3m-6}\cdot t(m)$.

Let $P_m$ denote the set of pairs $(G,c)$ such that $G$ is a spanning graph on $m$ vertices embedded in the sphere, and $c$ is a $2$-coloring of $G$'s faces in black and white such that every edge bounds exactly one white face and one black face. In what follows we will be interested in graphs in $P_m$, and we will assume implicitly that they come equipped with a $2$-coloring $c$. Note that the existence of such a $2$-coloring implies that the vertex degrees are all even.

Given such a graph $G$, let $r = r(G)$ be the number of its connected components, and let $F_b=F_b(G)$ and $F_w=F_w(G)$ be the sets of its white and black faces respectively. For convenience, we write $W = |F_w|$, $B = |F_b|$ and $F = F_w \cup F_b$. Let $l_f$ denote the number of connected components of $G$ touched by a face $f$.

\begin{prop}\label{prop:planar_estimates}
Let $G \in P_m$. There holds:
\begin{enumerate}
\item\label{it:planar_face_bound}
The number of faces of any color is at most $m-2.$

\item\label{it:conn_comps_bound}
The number of connected components of $G$ is at most $\lfloor\frac{m}{3}\rfloor.$

\item\label{it:it:estimation_of_l_f_sum}
\[
\sum_{f \in F}l_f-1\leq \frac{m}{3}.
\]
\end{enumerate}
\end{prop}
\begin{proof}

By Euler's formula $|F|-|E|+|V|=r+1$. In this case $|V|=m$ and $3|F|\leq 2|E|$, since every face is bounded by at least three edges, and every edge belongs to exactly two faces. Thus, $m-\frac{|E|}{3}\geq r+1$. Note that $3 W,3 B\leq |E|$, since each face is bounded by at least three edges, and every edge belongs to exactly one white and one black face.

For the second part, observe that all of the degrees in $G$ must be even. Since there are no isolated vertices, it follows that any connected component of $G$ is made of at least $3$ vertices, and so the number of connected components is at most $\lfloor\frac{m}{3}\rfloor.$

For the last part, consider the bipartite graph $H=(V_0\cup V_1,E_H),$ defined as follows. $V_0$ is the set of faces of $G$, $V_1$ is the set of connected components of $G$. To define the edges, choose a face of $G$ to be the outer face, stretching the sphere into a plane. Now the notion of a face surrounding a connected component of $G$ is well defined, and we place an edge between a face $f$ and every connected component that it surrounds. Let $d_f$ be the degree of $f$ in $H$, and note that $d_f=l_f-1$ for every face except the outer face, $f_0$, for which $d_{f_0}=l_{f_0}$. The degrees of the vertices in $V_1$ are all one. By part \ref{it:conn_comps_bound}, $|V_1|\leq\frac{m}{3}.$ Thus,
\[\sum_{f \in F}l_f-1 \leq \sum_{f \in F} d_f=|V_1|\leq\frac{m}{3}.\]
\end{proof}

The following lemma, due to McKay \cite{McKay}, is useful in simplifying some of our arguments.

\begin{lem} \label{lem:mckay}
Let $z_1, ... ,z_n \geq 0$, and let $G$ be a planar graph on $n$ vertices.

The sum $L(G):=\sum_{\{i,j\} \in E(G)}{\min\{z_i,z_j\}}$ is at most $3 \cdot\sum_{i}{z_i}$.
\end{lem}
We give the proof of \cite{McKay}.
\begin{proof}
Assume without loss of generality that $z_1 \geq ... \geq z_n$.
We will show that for any planar graph $G$,
\begin{equation} \label{eqn:planar}
L(G) \leq z_2 + 2 z_3 + 3 z_4 + 3 z_5 + ... + 3 z_n.
\end{equation}
We can write $L(G) = \sum_{i=1}^n{a_i z_i}$, where $a_i$ is the number of edges $\{i,j\}$ such that $i > j$. Note that $\sum_{i=1}^k{a_i}$ is the number of edges in the induced graph on the first $k$ vertices, and therefore by the planarity of $G$, $\sum_{i=1}^k{a_i} \leq 3k-6,$ when $k\geq 3.$ For $k=2$ this sum is bounded by $1$ and for $k=1$ it is $0.$

Let $b_1=0$, $b_2=1$, $b_3=2$, and $b_4 = ... = b_n = 3$, and note that $\{b_i\}$ is the unique sequence such that $\sum_{i=1}^k{b_i} = 3k-6$ for every $k\geq3,$ and in general it is the tight upper bound on the number of edges of a connected planar graph on $k$ vertices. Thus, for every $k~\sum_{i=1}^k{a_i} \leq\sum_{i=1}^k{b_i}$. We would like to show that
\[L(G)\leq \sum_{i=1}^n b_iz_i.\]
We show it by performing the following series of steps. For every $i$ such that $a_i>b_i$, there must be some $j<i$ such that $a_j<b_j$, and so subtracting $1$ from $a_i$ and adding it to $a_j$ will increase the value of $L(G)$. After a finite number of such steps, we will have $a_i \leq b_i$ for all $i$ and therefore $\sum_{i=1}^n{a_i z_i} \leq\sum_{i=1}^n{b_i z_i} $.
\end{proof}

Let $G$ be a connected graph in $P_m$, and choose a face $f \in F(G)$ to be the outer face. We are interested in the maximal value of $B-W$, not counting the outer face.

We consider two cases, depending on whether $f$ is white or black.

\begin{lem}
\label{lem:W-B}
\begin{itemize}
\item If $f$ is white, $B-W \leq \lfloor \frac{m-1}{2}\rfloor$.
\item If $f$ is black, $B-W \leq \lfloor \frac{m-5}{2}\rfloor$.
\end{itemize}
\end{lem}

\begin{proof}
We say that $G$ is a counterexample if it violates the inequality in the first item.
Consider the case that $W=0$, that is, there are no white faces apart from the outer face.
Assume that there exists such a counterexample, and let $G$ be such a graph in which the number of black faces is minimal.

Let $\tilde{G}$ be the graph whose vertices are the (black) faces of $G$, and in which two vertices are connected if the corresponding faces share a vertex. Note that two faces cannot share two vertices, because that would create a white face. Note further that $\tilde{G}$ is connected if and only if $G$ is connected, and has no cycles, as these would imply the existence of a white face in $G$. Therefore, $\tilde{G}$ is a tree. Let $f$ be the face corresponding to a leaf of the tree $\tilde{G}$. We obtain a new graph $G'$ from $G$ by deleting $f$ from $G$. We observe that $m' := V(G') \leq m-2$, and that the number $B'$ of black faces in $G'$ is equal to $B-1$. We have:
\[
B' = B-1 > \lfloor \frac{m-1}{2}\rfloor - 1 = \lfloor \frac{m'-1}{2}\rfloor,
\]
contradicting the minimality of $G$.

Assume now that $W>0$. Let $G$ be a counterexample in which the number of white faces is minimal. Let $w$ be a white face, and let $x=a_0,a_1, ... ,a_p$ be the vertices of $w$ in a counterclockwise ordering from the point of view of $w$. Let $w=w_0,b_0,w_1,b_1,...,w_l,b_l$ be the faces adjoining $x$, in a counterclockwise ordering from the point of view of $x$, where one of the white faces $w_i$ may be the outer face.
Our aim is to split $x$ into two vertices so that the resulting graph has one more vertex and one less white face than $G$.

Let $e_0$ and $e_1$ be the edges of $x$ that bound $w$, and let $e_2, ... ,e_{2l+1}$ be the remaining edges of $x$, in a counterclockwise ordering, so that $e_{2i},e_{2i+1}$ bound the face $w_{i}$ and $e_{2i+1},e_{2i+2}$ bound the face $b_{i}$. We split $x$ into two vertices $x_1, x_2$, where the edges $e_1,e_2$ are connected to $x_2$ and the remaining edges are connected to $x_1$, and we push $x_1,x_2$ slightly apart so that $w$ and $w_1$ are merged into one white area. Let $G'$ denote the resulting planar graph.

\begin{figure}
 \centering
 \includegraphics[width=145mm,height=55mm]{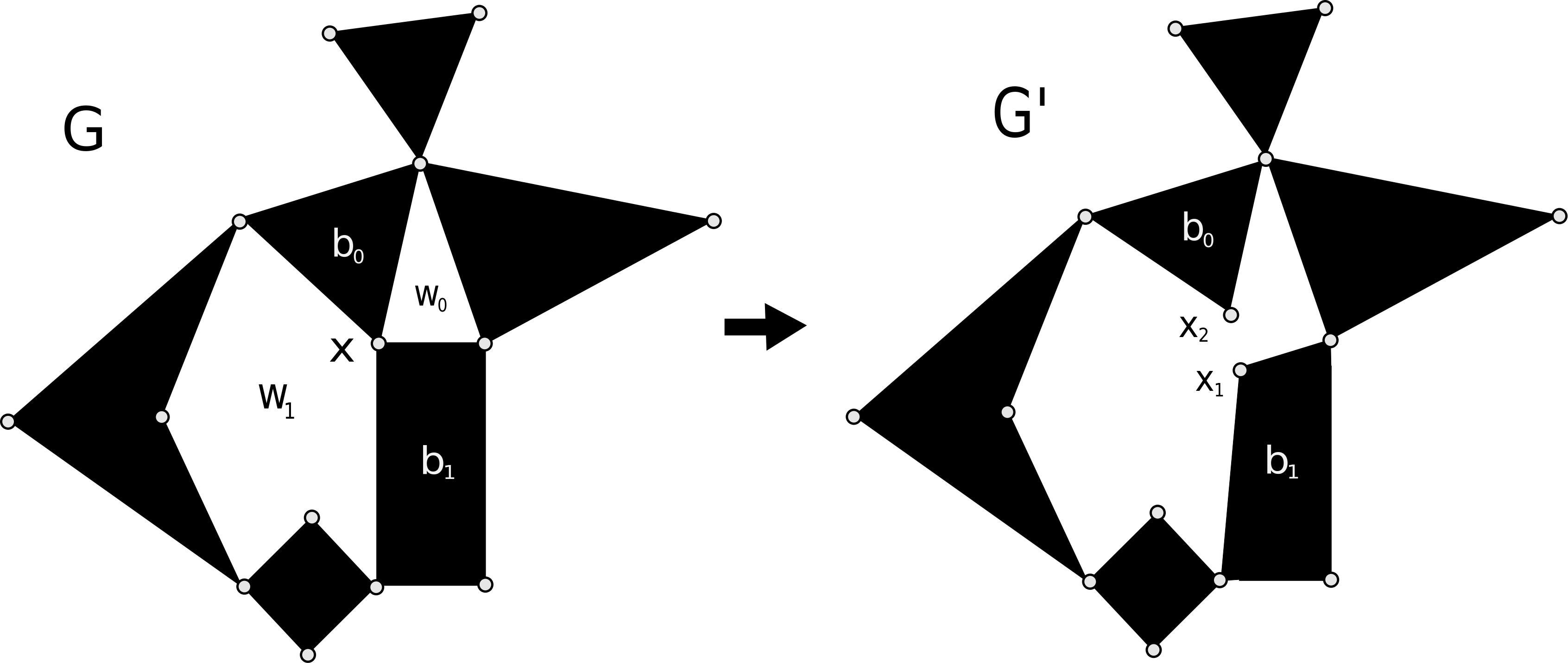}
 \caption{After the vertex $x$ is split into two vertices, $G'$ has one vertex more and one white face less than $G$.}
\end{figure}

Clearly, $G'$ has one vertex more and one white region less than $G$. Note that $G'$ is connected, as $x_1$ is connected to $x_2$ by the path $x_1, a_1,...,a_p,x_2$. Therefore, if $B',W',m'$ are the number of black faces, the number of white faces and the number of vertices in $G'$ respectively, then we have:
\[
B'-W' = B-W+1 > \lfloor \frac{m-1}{2} \rfloor + 1 \geq \lfloor \frac{m'-1}{2} \rfloor,
\]
which is a contradiction to the minimality of $G$.

For the second item of the lemma, note that we could have stated the first item in a more symmetrical way by considering a mapping of $G$ onto the sphere $S_2$. In this case, the outer face would be mapped to some face of the sphere, and we would have $B-W \leq \lfloor \frac{m-3}{2}\rfloor$, where here $W$ counts the white faces of $G$ including the outer face.

If the outer face is colored black, we can again consider a mapping of $G$ onto the sphere and obtain $B-W\leq \lfloor \frac{m-3}{2}\rfloor$ from the first part. Therefore, if we do not count the outer face, we get $B-W\leq \lfloor \frac{m-5}{2}\rfloor$ as desired.
\end{proof}

Let $G$ be a graph embedded in the sphere, and let $C_1,C_2,C_3$ be three different connected components of $G$. We say that $C_2$ \emph{separates} $C_1$ from $C_3$ if there is no continuous path from $C_1$ to $C_3$ on the sphere that does not intersect $C_2$. By a slight abuse of notation, we will sometimes say that two faces $f_1,f_2$ are separated by a connected component $C$ if there is no path from $f_1$ to $f_2$
 that does not intersect $C$.

Another concept that we will need is that of an automorphism of an embedded planar graph. The permutation group $\mathbb{S}_m$ acts on the set of labeled embedded planar graphs on $m$ vertices by relabeling their vertices. We define $Aut(G)$ to be the set of permutations $\pi \in \mathbb{S}_m$ such that $\pi.G$ is equivalent to $G$ via an isotopy.

\section{Triangulations} \label{sec:triangulations}

A triangulation of a surface is a graph embedded in the surface such that no two edges intersect and all of the faces are triangles. We are interested in triangulations in which the graph is simple. Sometimes we identify a triangulation and its set of triangles.
\begin{nn}
Let $T_{k,m_1,m_2,\ldots, m_l}$ be the number of triangulations of a surface of genus $0$ with $l$ boundary components of lengths $m_1,\ldots, m_l,~k$ internal labeled points, with no loops or double edges .
\end{nn}

\begin{prop}\label{prop:num_of_triangles}
\begin{enumerate}
\item\label{it:num_in_sphere}
A triangulation of a sphere on $n$ vertices has $2n-4$ triangles.
\item\label{it:num_in_disc_with_holes}
A triangulation of a sphere which has $l$ holes of lengths $m_1,\ldots,m_l$ and $k$ internal vertices has $2k+\sum_{i=1}^l m_i+2l-4$ triangles.
\end{enumerate}
\end{prop}
\begin{proof}
By Euler's formula, a graph embedded in a sphere satisfies $f-e+v=2,$ where $f,e,v$ are the numbers of faces, edges and vertices respectively.
For a triangulation on $n$ vertices, $v=n,~e=\frac{3}{2}f,$ thus $\frac{f}{2}=n-2$ and the first part follows.
For the second part, note that the $i$'th hole is a polygon with $m_i$ vertices, and so it can be triangulated using $m_i - 2$ triangles. If we then triangulate the remainder, the result is a triangulation of a sphere on $k+\sum_{i=1}^l m_i$ vertices, which by the first part uses $2(k+\sum_{i=1}^l m_i) - 4$ triangles. After subtracting the $\sum m_i - 2l$ that were added to fill the holes, the second part follows.
\end{proof}

The following theorem is based on techniques of Tutte \cite{Tutte}. See also \cite{GJ}, section 2.9.3.
\begin{theorem}\label{thm:Tutte}
The number $T_{k,m}$ of triangulations of a polygon with $m$ boundary vertices and $k$ labeled internal vertices is
\[\frac{2(2m-3)!(2m+4k-5)!}{(m-1)!(m-3)!(2m+3k-3)!}.\]
The number of triangulations of a sphere on $n$ labeled vertices, $|S_n|$, satisfies
\[
\lim_{n\to\infty}\frac{|S_n|}{n!\gamma^nn^{-7/2}}\to z,
\]
where $z$ is some constant.
\end{theorem}
The second part follows from the first one by noting that $(2n-4)|S_n|$, the number of sphere triangulations with a distinguished triangle, is $2\binom{n}{3}T_{n-3,3}.$ Now, up to a constant factor by the first part of the theorem, $T_{n-3,3}$ is
\[\frac{(6+4(n-3)-5)!}{(6+3(n-3)-3)!}\simeq n^{-5}n!\binom{4n}{3n}\simeq n^{-\frac{11}{2}} \gamma^n n!,\]
where $\simeq$ stands for asymptotic equality up to a multiplicative constant. The result for $|S_n|$ now follows.

\begin{lem}\label{lem:estimation_of_triangulations_of_sphere_with_holes}
Let $m_1,\ldots,m_l$ be positive integers whose sum is $M$, and let $k$ be a non negative integer. Then
\[\
\frac{T_{k,m_1,\ldots,m_l}}{k!\gamma^k}\leq c_2 ^M(k+M)^{l-\frac{7}{2}},
\]
for some constant $c_2$.
\end{lem}
\begin{proof}
Consider all triangulations of a polygon with $m_1$ boundary points, and \[K=k+\sum_{i=2}m_i +l-1\] labeled internal points.
On the one hand, the number of such triangulations is
\[T_{K,m} = \frac{2(2m_1-3)!(2m_1+4K-5)!}{(m_1-1)!(m_1-3)!(2m_1+3K-3)!}\]
by Theorem \ref{thm:Tutte}.
On the other hand, one can construct
\[\frac{(k+\sum_{i=2}m_i)(k+\sum_{i=2}^lm_i-1)\cdot (k+1)}{\prod_{i=2}^l m_i} T_{k,m_1,\ldots,m_l}\] different triangulations in the following way:
Let $p_1,\ldots, p_{l-1}$ be the last $l-1$ points. For $p_i$ choose $m_i$ neighbors from the remaining $k+\sum_{i\geq 2}m_i$ points, and order them cyclically. For every $2 \leq i\leq l$ put triangles between $p_i$ and every pair of consecutive neighbors of $p_i$. Then complete the triangulation to a whole polygon triangulation. There are $T_{k,m_1,\ldots,m_l}$ ways to perform the last step, and $\frac{(k+\sum_{i=2}m_i)(k+\sum_{i=2}m_i-1)\cdots (k+1)}{\prod_{i=2}^l m_i}$ to do the former.
We have
\[\frac{(k+\sum_{i=2}m_i)(k+\sum_{i=2}m_i-1)\cdot (k+1)}{\prod_{i=2}^l m_i} T_{k,m_1,\ldots,m_l}\leq\frac{2(2m_1-3)!(2m_1+4K-5)!}{(m_1-1)!(m_1-3)!(2m_1+3K-3)!},\] or equivalently,
\[ \frac{T_{k,m_1,\ldots,m_l}}{k!}\leq\frac{2(\prod_{i=2}^l m_i)(2m_1-3)!(2m_1+4K-5)!}{(k+\sum_{i=2}m_i)!(m_1-1)!(m_1-3)!(2m_1+3K-3)!}.\]
Write $M=\sum_{i=1}^l m_i.$
By standard estimations, the right hand side is no more than \[c^M\frac{{(2m_1+4(k+\sum_{i\geq 2}m_i))}^{{2m_1+4(k+\sum_{i\geq 2}m_i)}}}{(2m_1+3(k+\sum_{i\geq 2}m_i))^{2m_1+3(k+\sum_{i\geq 2}m_i)}(k+\sum_{i\geq 2}m_i)^{k+\sum_{i\geq 2}m_i}}(k+M)^{l-\frac{7}{2}},\]
for some constant $c$.
Consider now \[\frac{{(2m_1+4(k+\sum_{i\geq 2}m_i))}^{{2m_1+4(k+\sum_{i\geq 2}m_i)}}}{(2m_1+3(k+\sum_{i\geq 2}m_i))^{(2m_1+3(k+\sum_{i\geq 2}m_i))}(k+\sum_{i\geq 2}m_i)^{k+\sum_{i\geq 2}m_i}}\gamma^{-k}.\] It can be written as
\[\mu\left(\frac{{(2m_1+4(k+\sum_{i\geq 2}m_i))}}{4(k+\sum_{i\geq 2}m_i)}\right)^{k+\sum_{i\geq 2}m_i}
\left(\frac{3(2m_1+4(k+\sum_{i\geq 2}m_i))}{4(2m_1+3(k+\sum_{i\geq 2}m_i))}\right)^{2m_1+3(k+\sum_{i\geq 2}m_i)},\]
where \[\mu=\frac{4^{(2m_1+4(\sum_{i\geq 2}m_i))}}{3^{(2m_1+3(\sum_{i\geq 2}m_i))}}< 4^{4M}.\]
Let $x=k+\sum_{i\geq 2}m_i$. We would like to show that
\[\left(1+\frac{m_1}{2x}\right)^{x} \left(1-\frac{m_1}{4m_1+6x}\right)^{2m_1+3x}\leq c'^{m_1}.\]
This is equivalent to
\[f(\alpha)=(1+\alpha/2)^{\frac{1}{\alpha}} \cdot \left(1-\frac{1}{4\alpha+6}\right)^{2+3/\alpha}\leq c',~\alpha = m_1/x.\]
$f(\alpha)\geq 0$ for $\alpha\geq0,$ tends to $0$ when $\alpha\to 0,$ and to $1$ when $\alpha\to\infty,$ and thus it is bounded by some constant $c'$.

Collecting all of the above, we get
\[\frac{T_{k,m_1,\ldots,m_l}}{k!\gamma^k}\leq c_2^M(k+M)^{l-\frac{7}{2}},\] for some constant $c_2,$
as claimed.
\end{proof}

\section{Proof of Theorem \ref{thm:main}}\label{sec:proof}
The proof of Item \ref{it:first} of Theorem \ref{thm:main} is based on a simple first moment argument. Let $T$ denote the number of triangulations contained in $X$. Then $T = \sum_{s \in S_n}{T_s}$, where $T_s$ is the indicator random variable of the event that $s \subseteq X$. Therefore
\[
\mathbb{E}[T] = |S_n|\cdot p^{2n-4} = O\left(n! \gamma^n n^{-\frac{7}{2}} \cdot (1-\varepsilon)^{2n-4}\left(\frac{e}{\gamma n}\right)^{n-2}\right) = o(1).
\]
Markov's inequality now implies that $\Pr(T \geq 1) = o(1)$.

The main content of this paper is the proof of the Item \ref{it:second} of Theorem \ref{thm:main}. Our proof is based on the second moment method. For convenience we write
\[p=(1-\varepsilon)^{-1}\left(\frac{e}{\gamma n}\right)^{\frac{1}{2}}.\]
By Chebyshev's inequality, it suffices to show that
\[
\mathbb{E}[T^2] = (1+o(1)) \mathbb{E}[T]^2.
\]
Now,
\[
\mathbb{E}[T^2] = \sum_{s,s' \in S_n}{\Pr(s \subseteq X \text{ and } s' \subseteq X)} = p^{4n-8} \sum_{s,s' \in S_n}{p^{-|s \cap s'|}}.
\]
On the other hand, $\mathbb{E}[T]^2 = |S_n|^2 p^{4n-8}$ and so it suffices to prove that
\begin{equation}\label{eq:to_prove}\frac{1}{|S_n|^2}\sum_{s,s' \in S_n}{p^{-|s \cap s'|}} = 1+o(1).\end{equation}

The next step is to change the order of summation. We call a collection of triangles $F$ \textit{completable} if there is a triangulation in $S_n$ containing $F$, and we sum over all completable $F$'s:
\[
\frac{1}{|S_n|^2}\sum_{s,s' \in S_n}{p^{-{|s \cap s'|}}} = \frac{1}{|S_n|^2}\sum_{F}{p^{-|F|}\cdot|\{(s,s'):F = s \cap s'\}|}.
\]
The following lemma is reminiscent of Lemma 9 from \cite{Rior}. It states that we can pay a small penalty and make the weaker assumption that $s\cap s'$ contains $F$, rather than $F = s \cap s'$.
\begin{lem}
\begin{align*}\sum_{F}{p^{-|F|}|\{(s,s'):F = s \cap s'\}|} \\
= \sum_{F}{\left(\frac{1}{p}-1\right)^{|F|}|\{(s,s'):F \subseteq s \cap s'\}|}.
\end{align*}
\end{lem}
\begin{proof}
We change the order of the summation.
\begin{align*}
\sum_{F}{\left(\frac{1}{p}-1\right)^{|F|}|\{(s,s'):F \subseteq s \cap s'\}|} \\
= \sum_{(s,s')}\sum_{F:F \subseteq s \cap s'} {\left(\frac{1}{p}-1\right)^{|F|}} \\
=\sum_{(s,s')}\sum_{i=0}^{|s \cap s'|}{\binom{|s \cap s'|}{i}}\left(\frac{1}{p}-1\right)^i \\
= \sum_{(s,s')}\left(\frac{1}{p}\right)^{|s \cap s'|} \\
= \sum_{F}{p^{-|F|}|\{(s,s'):F = s \cap s'\}|}.
\end{align*}
\end{proof}

In our case $p = o(1)$, and so the ratio between $\left(\frac{1}{p}\right)$ and $\left(\frac{1}{p} - 1\right)$ is $(1+o(1))$. Therefore by changing $\varepsilon$ slightly we can compensate. That is, it suffices to show that for any $\varepsilon>0$ we have
\begin{equation*}
\frac{1}{|S_n|^2}\sum_{F}{p^{-|F|}|\{(s,s'):F \subseteq s \cap s'\}|} = 1 + o(1),
\end{equation*}
where the sum is taken over all completable $F$'s.

Next, we deal with the extremal cases in which $F$ is either the empty set or is a complete triangulation of the sphere. Clearly, the summand that corresponds to $F=\phi$ is one. On the other hand, the summand that corresponds to $F \in S_n$ is $o(1)$ because
\[
\sum_{s \in S_n}{p^{-(2n-4)}} = |S_n|^2 \cdot \left(|S_n| p^{2n-4}\right)^{-1} = o(1) \cdot |S_n|^2.
\]

Thus, it suffices to show that
\begin{equation}\label{eq:intermediate}
\frac{1}{|S_n|^2}\sum_{F}{p^{-|F|}|\{(s,s'):F \subseteq s \cap s'\}|} = o(1),
\end{equation}
where the sum is taken over all nonempty completable sets $F$ that are not in $S_n$.

\subsection{Step one: Reduction to good graphs}\label{subsec1}
Any completable $F$ can be embedded into the sphere. Once we fix such an embedding, one can complete $F$ to a triangulation in $S_n$ by filling in the uncovered areas of the sphere with triangles. However, some care is needed, as there may be many different, nonequivalent embeddings of $F$.

We say that two embeddings $\eta_1, \eta_2$ of $F$ into the sphere are equivalent if $\eta_2$ is obtained from the composition of $\eta_1$ with an isotopy.

Let $\text{Em}(F)$ denote the set of embeddings of $F$ into the sphere up to isotopy, 
and let $\pi(F) = |\text{Em}(F)|$ be the number of such embeddings. Any $s \in S_n$ such that $F \subseteq s$ induces a unique embedding of $F$ into the sphere, which we denote $s(F)$.
We have
\begin{align*}
&\frac{1}{|S_n|^2}\sum_{F} p^{-|F|}\{(s,s'):F \subseteq s \cap s'\}|  \\
& = \frac{1}{|S_n|^2}\sum_{F} p^{-|F|} \sum_{\eta_1,\eta_2 \in \text{Em}(F)}{|\{s: F \subseteq s, s(F) = \eta_1\}| \cdot|\{s: F \subseteq s, s(F) = \eta_2 \}| }.
\end{align*}

We wish to pinpoint a small set of properties of $F$ that determines the value of a summand in the above sum.

We define the \textit{boundary} $\partial F$ of $F$ to be the set of edges contained in a single triangle of $F$, and we say that a vertex is an interior vertex of $F$ if it is contained in a triangle of $F$, but not in a boundary edge. The boundary of $F$ is a planar graph, and any embedding of $F$ into the sphere induces a face structure on $\partial F$, together with a 2-coloring of its faces: color a face white if it is covered by triangles in $F$ and black otherwise.

Define the \emph{edge-connected} components of a completable $F$ to be the maximal subsets $C \subset F$ such that for any two triangles $t,t' \in C$ there is a sequence $t=t_0, ... ,t_l=t'$ with every two consecutive triangles sharing an edge. Note that in any embedding of $F$ into the sphere, the white faces of $\partial F$ correspond to edge-connected components of $F$. Thus, $F$ determines part of the face structure of the embedded graph $\partial F$: it determines the structure of the white faces. It also determines the number of interior points in each white face.

Define the \emph{vertex-connected} components of $F$ to be the maximal subsets $C \subset F$ such that for any two triangles $t,t' \in C$ there is a sequence $t=t_0, ... ,t_l=t'$ with every two consecutive triangles sharing a vertex. Let $C_W(F)$ denote the set of vertex-connected components. We will sometimes just call these the \emph{connected components} of $F$.

The properties of $F$ that we consider are the planar graph $\partial F$, the boundary of its white faces, and the allocation of interior points for its white faces. We call such a triple a \emph{completable filled planar graph (CFPG)} and denote the set of CFPGs whose total number of points is at most $n$ by $\tilde{P}_n$.  Note that the number of ways to complete an embedded completable $F$ to a triangulation $s \in S_n$ depends only on the CFPG of $F$.
We denote the set of completable $F$'s with a given CFPG $H$ by $M(H)$. Note that every $F \in M(H)$ has the same number $T_H$ of triangles, and the same number $|C_W|$ of connected components. We define $C_W(H):=C_W(G)$ for some $G \in M(H)$. Equivalently, $C_W(H)$ is the set of connected components defined by the white faces of $H$. We sometimes call these the white components of $H$.

Let $\text{Em}(H)$ denote the set of embeddings of a CFPG $H$ into the sphere up to isotopy.
~Now, since an embedding $\eta \in \text{Em}(H)$ determines uniquely an embedding for any $F \in M(H)$ that agrees with $\eta$, the number of ways to complete an embedding of $F$ into a triangulation of the sphere depends only on $H$ and $\eta$. This is the number of ways to allocate interior vertices to the black faces, and to triangulate them. We denote this number by $N(H,\eta)$.

We have
\begin{align*}
&\frac{1}{|S_n|^2}\sum_{F} p^{-|F|} \sum_{\eta_1,\eta_2 \in \text{Em}(F)}{|\{s: F \subseteq s, s(F) = \eta_1\}| \cdot|\{s: F \subseteq s, s(F) = \eta_2 \}| } \\
& = \frac{1}{|S_n|^2} \sum_{H \in \tilde{P}_n, \eta_1,\eta_2 \in \text{Em}(H)}{p^{-T_H} \cdot |M(H)| \cdot N(H,\eta_1) \cdot N(H,\eta_2) }.
\end{align*}

We denote the contribution of a given summand by
\[
f(H,\eta_1,\eta_2) := p^{-T_H} \cdot |M(H)| \cdot N(H,\eta_1) \cdot N(H,\eta_2).
\]
To show that this sum is $o(1)$, we will first reduce it to a sum over a smaller set of graphs, which will be called \emph{good graphs}, by showing that the dominant contribution to the sum comes from these graphs. We then carefully analyze the contribution of these good graphs.

A \emph{banana} is an embedded face $B$ of a planar graph with the following properties: $B$ has a single boundary component which is a $4$-gon, with vertices $v_1,\ldots,v_4,$ in this cyclic order, such that $v_2,v_4$ have edges only to $v_1,v_3,$ while $v_1,v_3$ are also connected in the graph obtained by erasing the edges of $B$. We call the pair $\{v_1,v_3\}$ the \emph{special pair} of the banana.
A configuration with many bananas may be tricky to handle, as it may have many embeddings.

As a motivating example, consider the family of CFPGs consisting of $2$ vertices $x_0,x_n$, together with bananas $B_1, ... ,B_{(n/2) - 1}$, such that the boundary of $B_i$ is $x_0,x_{2i-1},x_n,x_{2i}$. The number of ways to embed such a CFPG into the sphere is quite large: It is $(n/2-2)!2^{n/2}$, because it is equivalent to choosing a cyclic ordering of the bananas, as well as an orientation for each banana.
The goal of this section is to deal with this complication, by showing that graphs without too many bananas give the dominant contribution to the sum we want to estimate.

\begin{figure}
 \centering
 \includegraphics[width=145mm,height=55mm]{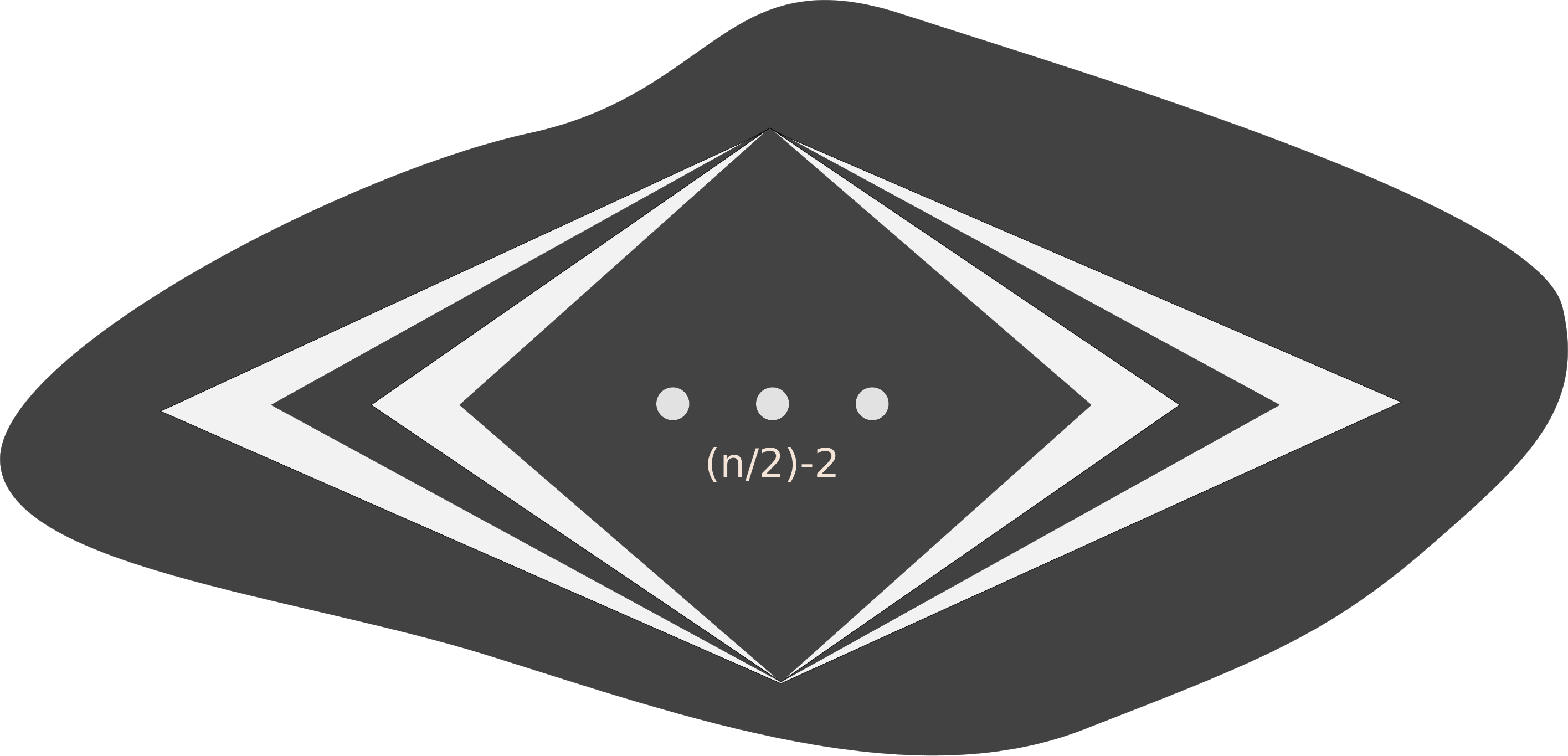}
 \caption{The contribution of this graph to the sum is only exponential in $n$, the number of vertices. A ``good" graph is one without too many of these white bananas. }
\end{figure}

The first step towards achieving this goal is to reduce to the case there is at most one black and at most one white banana with interior points.


Since the vertices are labeled, we can choose a fixed ordering of the set of all possible squares. We denote the minimal white banana in $H$ by $w(H),$ and, given an embedding $\eta \in \text{Em}(H),$ we denote the minimal black banana by $b(H,\eta).$

A CFPG $H$ is \emph{almost good} if all of the white bananas, except possibly $w(H)$, have no interior points. 

A \emph{good triangulation} of an embedded planar graph is a triangulation where in all bananas without internal points, the non special pair of vertices is connected by an edge. For an almost good CFPG $H$, let $M'(H)$ be the set of good triangulations of $H$'s white faces, and let $N'(H,\eta)$ be the number of good triangulations of the black faces in which all of the black bananas except for $b(H,\eta)$ have no interior points. Finally, for an almost good CFPG $H$ and embeddings $\eta_1,\eta_2,$ write,
\[
f'(H,\eta_1,\eta_2) := p^{-T_H} \cdot |M'(H)| \cdot N'(H,\eta_1) \cdot N'(H,\eta_2).
\]
Clearly $f'\leq f.$

For a given almost good CFPG $H$ and $\eta_1,\eta_2 \in Em(H)$, let $\tilde{r} = \tilde{r}(H,\eta_1,\eta_2)$ be the maximum between the number of white and black bananas over the two embeddings. Clearly, $\tilde{r}$ is independent of the interior point allocation.
Let $R(H)$ denote the set of CFPGs that can be obtained from $H$ by redistributing the interior points of $w(H)$ among all of the white bananas of $H$.

\begin{lem}\label{lem:almost_good}
Let $H$ be an almost good CFPG, all of whose white bananas have no interior points except possibly for the square $w(H)$.  Then
\[
\sum_{H' \in R(H),\eta'_1,\eta'_2 \in \text{Em}(H')}{ f(H',\eta'_1,\eta'_2)} \leq \sum_{\eta_1,\eta_2 \in \text{Em}(H)}{ 8^{3\tilde{r}(H,\eta_1,\eta_2)} \cdot f'(H,\eta_1,\eta_2)}.
\]
\end{lem}

\begin{proof}
The lemma is a consequence of the following proposition, whose proof can be found in the appendix.

\begin{prop}\label{prop:moving_all_points_of_bananas}
The number of ways to triangulate $r$ different $4-$gons using $k$ internal points is no more than $8^r$ times the number of ways to triangulate a single $4-$gon using $k$ internal points.
\end{prop}

Indeed, this proposition implies that $\sum_{H' \in R(H)} |M(H')| \leq 8^{\tilde{r}} |M'(H)|$, and that for any embedding $\eta$ we have
$N(H,\eta) \leq 8^{\tilde{r}} N'(H,\eta)$.

Therefore,
\begin{align*}
& \sum_{H' \in R(H),\eta_1,\eta_2 \in \text{Em}(H')}{ p^{-T_{H'}} \cdot |M(H')| N(H',\eta_1) N(H',\eta_2)} \leq \\
& p^{-T_{H}} \cdot
\left(\sum_{H' \in R(H)}{ |M(H')|}\right) \cdot
\left( \sum_{\eta_1,\eta_2 \in \text{Em}(H)}{ N(H,\eta_1) N(H,\eta_2) } \right) \leq \\
& 8^{3 \tilde{r}} \cdot \sum_{\eta_1,\eta_2 \in \text{Em}(H)} p^{-T_{H}} \cdot |M'(H)| N'(H,\eta_1) N'(H,\eta_2).
\end{align*}

\end{proof}

We now wish to define \emph{good} CFPGs and an erasure process from an almost good CFPG and a pair of embeddings to a good CPFG with a pair of embeddings.

Let $H$ be an embedded CFPG, and let $x$ be a vertex of $H$.
A \textit{local face} of $H$ at $x$ is a connected region of the intersection of a face of $x$ with a small punctured neighborhood at $x$. Note that a single face may have several local faces at $x$.
A local face may be white or black depending on the color of the corresponding face.

Let $w_1, ... ,w_l$ be the local white faces of $H$ at $x$, and let $e_0,...,e_{2l-1}$ be $x$'s edges, ordered so that $e_{2i},e_{2i+1}$ belong to the boundary of $w_i$.

We define the \emph{trimming} of $H$ at $x$ to be the CFPG obtained by splitting $x$ into $l$ vertices $x_1, ... ,x_l$ and moving them slightly apart, so that the only local white face touched by $x_i$ is $w_i$.

We denote the trimming at a set of vertices $A$ by $H_A$, and when $A$ is either a singleton $\{x\}$ or a pair $\{x,y\}$ we write $H_x,H_{x,y}$ respectively.

A \emph{leaf} of $x$ is a connected component of $H_x$ which is not a connected component of $H,$ and a \emph{branch} between $x,y$ is a connected component of $H_{x,y}$ which is not a connected component of $H$ or a leaf at $x$ or $y.$

Let $\tilde{L}(x)$ be the number of leaves of $x,$ and $\tilde{L}(x,y)$ be the number of branches of $x,y$. We call all of the white bananas between $x,y$, except for the minimal banana between them (in the ordering), \textit{extra bananas}.

Denote by $\lozenge(H)$ the total number of extra bananas in $H$.
We will say that a CFPG $H$ on $m$ vertices is \emph{good} if it is almost good, and
\begin{equation} \label{def:good}
\lozenge(H)\leq \frac{24m}{49}
\end{equation}

Our goal is now to show that we may, without loss of generality, sum over only good CFPGs.
Let $H'$ be any (not necessarily good) CFPG, and let $\eta'_1,\eta'_2 \in \text{Em}(H')$. An \textit{erasure step} on $H', \eta'_1,\eta'_2$ is the following modification of $H'$ and $\eta'_1,\eta'_2$: Choose a special pair $x,y\in[n],~x<y.$ 
Choose now an extra white banana $(x,a,y,b)$ containing no interior points, such that in both embedded copies of $H'$, the black face to its right\footnote{'right' ('left') is defined using the clockwise (counter clockwise) orientation at $x$ starting from the diagonal $\overline{xy}$ of the banana.} is a black banana containing no interior points, and, by the definition of a banana, also no connected components of $H$, or leaves of $x,y$. We erase this banana, and move the vertices $a$ and $b$ to the minimal white banana $w(H)$. We get a new CFPG $H$, in which the pair $x,y$ has one banana less, and the banana $w(H)$ has two additional interior points. We also get two embeddings $\eta_1$ and $\eta_2$, which are induced from $\eta'_1$ and $\eta'_2$.The \emph{erasure process} is the result of performing iteratively erasure steps, until there are no more possible steps\footnote{During the erasure process black bananas may be united. We choose $b(H,\eta)$ in a way which is stable under these changes: The minimal black banana is the black banana whose special pair $\{x,y\}$ is minimal in lexicographic order, and, among the bananas with special pair $\{x,y\}$ whose left non-special vertex $z,$ is minimal.}.

For a \emph{good} CFPG $H$ and $\eta_1,\eta_2 \in \text{Em}(H)$, we say that a triple $(H',\eta'_1,\eta'_2)$ \emph{reduces} to  $(H,\eta_1,\eta_2)$ if $H$ can be obtained from $H'$ by the erasure process. Let
$A(H, \eta_1,\eta_2)$ be the set of all triples $(H',\eta'_1,\eta'_2)$ which can be reduced $(H,\eta_1,\eta_2)$. Write $A_r(H,\eta_1,\eta_2)$ for the set $(H',\eta'_1,\eta'_2)\in A(H, \eta_1,\eta_2)$ such that $H'$ has $r$ more white bananas then $H.$

\begin{obs} \label{obs:small_difference}
Let $H$ be a good CFPG on $m$ vertices, let $\eta_1,\eta_2 \in Em(H)$, and let $(H',\eta'_1,\eta'_2) \in A_r(H,\eta_2,\eta_2)$. Then $|r-\tilde{r}(H',\eta'_1,\eta'_2)| \leq m/2$.
\end{obs}

Indeed, a step of the erasure process does not change the difference between the number of white and black bananas. Therefore $|r-\tilde{r}(H',\eta'_1,\eta'_2)|$ is at most $\tilde{r}(H,\eta_1,\eta_2)$, and this is no more than $m/2$.

The following lemma states that every CFPG can be reduced to some good CFPG.
\begin{lem}\label{lem:good}
For every almost good CFPG $H'$ and pair of embeddings $\eta'_1,\eta'_2 \in \text{Em}(H')$, there is a good CFPG $H$ and a pair of embeddings $\eta_1,\eta_2 \in \text{Em}(H)$ such that $(H',\eta'_1,\eta'_2) \in A(H, \eta_1,\eta_2)$.
\end{lem}

\begin{proof}
We must show that the erasure process always ends at a good CFPG.
Suppose $(H,\eta_1,\eta_2)$ is the result of the erasure process applied to $(H',\eta'_1,\eta'_2)$, and let $m = |V(H)|$.
Let $H_0$ be the graph obtained from $H$ by erasing all of the extra bananas. $H_0$ has $v_0=m-2\lozenge(H)$ vertices and \[e_0\leq 3(v_0-2|C_W|)\] edges, by a standard bound for planar graphs. Note that the connected components of $H_0$ correspond to those of $H$, and that there is a natural bijection between leaves of $H_0$ and leaves of $H$ and between branches of $H_0$ between points $x,y$ which are not extra bananas, and branches of $H$ between the same pair of points which are not extra bananas.

For $i=1,2$ write $\lozenge_\emptyset^i$ for the set of extra bananas $b$ in $H$ such that under $\eta_i,$ to the right of $b$ there is a black banana which is not the special one (the one that will contain internal points). Let $\lozenge_{\neq\emptyset}^i$ be its complement.
By the definition of the erasure process, every extra banana in $H$ must belong to either $\lozenge_{\neq\emptyset}^1$ or $\lozenge_{\neq\emptyset}^2$. Therefore,
\begin{equation}\label{eq:lozenges_sum_to}
 \lozenge\leq|\lozenge_{\neq\emptyset}^1|+|\lozenge_{\neq\emptyset}^2|.
\end{equation}
A pair of vertices, $\{i,j\}$ is \emph{optional in $H~(H_0)$} if $i$ and $j$ belong to the same white banana of $H~(H_0)$, their distance on this banana is exactly $2,$ and if the banana is not a connected component of $H~(H_0)$ then $i,j$ are its special pair.
Let $Opt(H)~(Opt(H_0))$ be the set of optional pairs, and note that $Opt(H)\subseteq Opt(H_0).$ In addition, if $H$ is obtained from some $H'$ as above, then during the erasure process only bananas between an optional pair can be erased.
Let $X$ be the set whose elements are the special black banana, the connected components in $C_W(H_0)$, the leaves of $H_0$ and the branches of $H_0$ between pairs of points $\{x,y\}\in Opt(H_0).$
We will define for $i=1,2$ injections
\[h_i:\lozenge_{\neq\emptyset}^i\to X,\] as follows.
Modify $\eta_i$ to a planar embedding by choosing a point to be a point at infinity.  Under the refined $\eta_i,$ for $b\in\lozenge_{\neq\emptyset}^i$ with a special pair $x,y$ there are precisely the following possibilities:
\begin{enumerate}
\item The black face to the right of $b$ is the special black banana. In this case define $h_i(b)$ to be this banana.
\item If we erase all connected components not containing $b,$ and all leaves at $x,y$ not containing $b$ the induced face to the right of $b$ is a black banana. In this case before this erasure there was at least one leaf of $x$ or $y$ or a connected component which touches the black face to the right of $b.$ If there are such leaves, define $h_i(b)$ to be one of them, considered as a leaf of $H_0$. If there are only connected components then there are two possibilities. If the boundary of the face to the right of $b$ separates these components and infinity, we define $h_i(b)$ to be one of these components, considered as a component of $H_0.$ Otherwise, $h_i(b)$ is the minimal banana between $x,y,$ considered as a branch in $H_0.$
\item Otherwise the face to the right of $b$ is bounded by another branch between $x,y$ which is not a banana. This branch corresponds to a unique branch of $H_0$ between $x,y$ (which must belong to $Opt(H_0)$). Let $h_i(b)$ be that branch.
\end{enumerate}
By definition, if $b_1,b_2$ are two bananas from the same special pair, under any embedding $\eta$ the face to the right of $b_1$ will not be the face to the right of $b_2.$ Thus, the special black banana, each leaf and each branch which is not a banana may appear at most once in the image of $h_i.$ 
Furthermore, planarity implies that a connected component and a white banana (which is minimal for its special pair) may appear at most once. Indeed, a connected component is separated from infinity by a single black face which touches it, or not separated at all. A minimal banana between $x,y$ may appear in the image of $h_i$ only if the outward boundary of the connected component containing this banana is made of two bananas between $x,y.$ Only one of them, $b',$ may have that $h_i(b')$ is the minimal banana.
Thus, for each $i~h_i$ is an injection, and hence:
\begin{equation}\label{eq:bounds_lozenge_nem_i}
|\lozenge_{\neq\emptyset}^i|\leq |X|\Rightarrow\lozenge\leq 2+2|C_W|+2|Leaf(H_0)|+2|Branch(H_0)|,
\end{equation}
where $Leaf(H_0)$ is the set of leaves of $H_0,$ $Branch(H_0)$ is the set of branches of $H_0$ between optional pairs of vertices $\{x,y\}\in Opt(H_0),$
and we have used \eqref{eq:lozenges_sum_to} to derive the right inequality.

The number of leaves at a vertex $x$ of $H_0$ is bounded by $\frac{1}{2}deg_{H_0}(x).$ Thus, \[|Leaf(H_0)|\leq \sum_{x \in V(H_0)}{\frac{1}{2}deg_{H_0}(x)}=e_0.\]
Similarly, the number of branches between $x,y$ is bounded by $\frac{1}{2}\min(deg_{H_0}(x),deg_{H_0}(y)).$
Let $G$ be the graph on the vertices of $H_0$ whose edge set is $Opt(H_0)$. This new graph $G$ is planar, as can be seen by taking an embedding of $H_0,$ drawing the diagonals of the bananas which correspond to $Opt(H_0)$ and erasing the remaining edges of $H_0.$ Thus by Lemma \ref{lem:mckay} we have
\begin{equation}\label{eq:branch_bound}
|Branch(H_0)| \leq \frac12 \sum_{i,j \in Opt(H_0)}{\min(deg_{H_0}(i),deg_{H_0}(j))}\leq \frac32 \sum_{i} deg_{H_0}(i) \leq 3e_0.
\end{equation}
Putting this together, we see that
\[\lozenge\leq 2 + 2|C_W| + 8e_0\leq 24(m-2\lozenge),\]
and therefore
\[\lozenge\leq\frac{24m}{49},\]
as claimed.
\end{proof}

The following lemma is the main technical proof in this section. It states that for any good CFPG $H$, the sum of the contributions of all the triplets that can be reduced to it, is not much larger than the contribution of $H$ itself.
\begin{lem} \label{lem:good_bound}
There is a global constant $C>0$ such that for any good CFPG $H$ on $m$ vertices, and pair of embeddings $\eta_1,\eta_2 \in \text{Em}(H)$, we have
\[
\sum_{(H',\eta'_1,\eta'_2)\in A_r(H,\eta_1,\eta_2)}{f'(H',\eta'_1,\eta'_2)} \leq C^m 8.1^{-3r} \cdot f'(H,\eta_1,\eta_2).
\]
\end{lem}

\begin{proof}
Let $H$ be a good graph on $m$ vertices. We use the terminology of an optional pair, from the proof of Lemma \ref{lem:good}. 
The number of optional pairs for $H$ is no more than twice the total number of bananas, hence no more than $m$.
If $H$ is a result of the erasure process then $H$ has, in particular, a distinguished banana $w(H)$ which has $k_0$ internal points.

The set of graphs $H'$ which can be reduced to $H$ in the erasure process and have $r$ extra bananas is of cardinality no more than

\[\binom{k_0}{2r}\sum_{\{r_{\{i,j\}}\}_{\{i,j\}\in Opt(H)}:\sum r_{\{i,j\}}=r}\binom{2r}{\{2r_{\{i,j\}}\}_{\{i,j\} \in Opt(H)}}\prod_{\{i,j\}\in Opt(H)}(2r_{\{i,j\}}-1)!!\]

The first binomial is the number of ways to choose the points that will create the $r$ new bananas. The sum $\{r_{\{i,j\}}\}$ tell us how many extra bananas we are adding to each optional pair. The next multinomial chooses which of the $2r$ points go to which optional pair. Finally, the factor $(2r_{\{i,j\}}-1)!!$ is the number of ways to pair up the $r_{\{i,j\}}$ points for a given pair $i,j$.

Recall that $\tilde{L}(i,j)$ is the number of branches between $i$ and $j$.
Given such $H',H,$ and an embedding $\eta$ of $H,$ the number of embeddings of $H'$ which may give rise to $(H,\eta)$ at the end of the erasure process is no more than
\[\prod_{\{i,j\}\in Opt(H)}\frac{(r_{\{i,j\}}+\tilde{L}(i,j)-1)!}{(\tilde{L}(i,j)-1)!},\]
as these embeddings correspond to cyclic orderings of all the branches between $i,j$ in $H'$ that agree with the cyclic orders given by $\eta$ to the branches between $i,j$ in $H$.

Thus, for a good $H$ as above, the ratio
\begin{equation} \label{eqn:ratio}
\frac{\sum_{(H',\eta'_1,\eta'_2)\in A_r(H,\eta_1,\eta_2)}{f(H',\eta'_1,\eta'_2)}}{f(H,\eta_1,\eta_2)}
\end{equation}
is bounded by
\[
\frac{\binom{k_0}{2r} \cdot T_{k_0-2r,4}}{{T_{k_0,4} \cdot p^{-2r}}}\cdot
\sum_{\{r_{\{i,j\}}\}_{\{i,j\}\in Opt(H)}:\sum r_{\{i,j\}}=r}\binom{2r}{\{2r_{\{i,j\}}\}} \cdot
\]
\[
 \cdot \prod_{\{i,j\}\in Opt(H)}(2r_{\{i,j\}}-1)!!   \prod_{\{i,j\}\in Opt(H)}\left(\frac{(r_{\{i,j\}}+\tilde{L}(i,j)-1)!}{(\tilde{L}(i,j)-1)!}\right) ^2    ,
\]
where the power of $p$ counts the difference in white triangles, since $T_H = T_{H'}+2r$. Our goal is now to estimate this sum.

By Theorem \ref{thm:Tutte},
\[\frac{T_{k,4}}{(1+k)^{-\frac{5}{2}}\gamma^k k!}=O(1).\]

Note that the cardinality of the set $\{\{r_{\{i,j\}}\}_{\{i,j\}\in Opt(H)}:\sum r_{\{i,j\}}=r\}$ is $\binom{r+|Opt(H)|-1}{|Opt(H)|-1}$, which by Observation \ref{obs:useful_binom_est} is bounded for any positive $\delta$ by \[C_\delta^{|Opt(H)|}(1+\delta)^r\leq C_\delta^{m}(1+\delta)^r.\]

Also by Observation \ref{obs:useful_binom_est}, the expression $ \prod_{\{i,j\}\in Opt(H)}\left(\frac{(r_{\{i,j\}}+\tilde{L}(i,j)-1)!}{(\tilde{L}(i,j)-1)!}\right)$ is bounded by
\[C_\delta^{\sum \tilde{L}(i,j)}(1+\delta)^r\prod_{\{i,j\}}(r_{\{i,j\}})!\]

Observe that the number of branches between any two vertices $i,j$ is upper bounded by $\frac12\min(deg_H(i),deg_H(j))$. The exact same argument used to show \eqref{eq:branch_bound} yields
\[
\sum_{i,j \in Opt(H)} \tilde{L}(i,j) \leq \frac12 \sum_{i,j \in Opt(H)}{\min(deg_H(i),deg_H(j))} \leq\frac32 \sum_{i} deg_H(i) \leq 9m.
\]

Using these estimates, and the fact that $(2a-1)!!\leq a!2^a$, \ref{eqn:ratio} is bounded by
\[
\frac{p^{2r} (1+k_0)^{5/2}}{\gamma^{2r} (1+k_0-2r)^{5/2}}C_{\delta}^{10m}(1+\delta)^{3r} \cdot r!2^r.
\]
Since $2r\leq k_0,$ for any $\delta>0$ one can find $c_\delta$ with
\[\left(\frac{1+k_0}{1+k_0-2r}\right)^{2.5}\leq c_\delta(1+\delta)^r,\]
putting everything together, plugging in the value of $p$ and using Stirling's estimate,
\ref{eqn:ratio} is bounded by
\[ C_{\delta}^{10m} \left(\frac{2r(1+\delta)^4}{\gamma^3 n}\right)^r. \]
As $\gamma>8.1$ and using the fact that $r \leq n/2$, by choosing a sufficiently small $\delta$, this is at most
\[C^m 8.1^{-3r}.\]
\end{proof}

Putting Lemmas \ref{lem:almost_good},\ref{lem:good}, and \ref{lem:good_bound} together, we have
\[
\sum_{H\in \tilde{P}_n,\eta_1,\eta_2 \in Em(H)}{f(H,\eta_1,\eta_2)} =
\]
\[
\sum_{H \text{ is almost good.}} \sum_{H' \in R(H), \eta'_1,\eta'_2 \in Em(H)}{f(H',\eta'_1,\eta'_2)} \leq
\]
\[
\sum_{H \text{ is almost good, }\eta_1,\eta_2 \in Em(H)}{8^{3\tilde{r}} \cdot f'(H,\eta_1,\eta_2)} \leq
\]
\[
\sum_{m \geq 3}\sum_{H \text{ is good on $m$ vertices, } \eta_1,\eta_2 \in Em(H)}  \sum_{r\geq 0} \sum_{(H',\eta'_1,\eta'_2) \in A_r(H,\eta_1,\eta_2)}{8^{3\tilde{r}} \cdot f'(H',\eta'_1,\eta'_2)} \leq
\]
\[
\sum_{m \geq 3}  \sum_{H \text{ is good on $m$ vertices, } \eta_1,\eta_2 \in Em(H)} \sum_{r\geq 0} 8^{3(r+m/2)} \cdot \sum_{(H',\eta'_1,\eta'_2) \in A_r(H,\eta_1,\eta_2)}{f'(H',\eta'_1,\eta'_2)}\leq
\]
\[
\sum_{m \geq 3} C^m  \sum_{H \text{ is good on $m$ vertices, } \eta_1,\eta_2 \in Em(H)} f'(H,\eta_1,\eta_2) \cdot \left(  \sum_{r\geq 0} {\left(\frac{8}{8.1}\right)^{3r}}\right) .
\]
Here the first inequality follows from Lemma \ref{lem:almost_good}, the second follows from Lemma \ref{lem:good}, the third follows from observation \ref{obs:small_difference}, and the fourth follows from Lemma \ref{lem:good_bound}.

Therefore, we see that Theorem \ref{thm:main} will follow from showing that for any constant $C>0,$
\begin{align} \label{eq:to_prove_good_graphs}
\frac{1}{|S_n|^2}\sum_{m \geq 3} C^m\sum_{H~\text{a good CFPG on $m$ vertices}}\sum_{\eta_1,\eta_2\in Em(H)}f'(H,\eta_1,\eta_2) = o(1).
\end{align}

\subsection{Step two: Bounding the number of ways to embed $H$}\label{subsec2}
The contribution of  a CFPG $H$ to the sum in (\eqref{eq:to_prove_good_graphs}) is
\[
\sum_{\eta_1,\eta_2 \in \text{Em}(H)} p^{-T_H} \cdot |M'(H)|\cdot N'(H,\eta_1) N'(H,\eta_2).
\]

By the Cauchy-Schwarz inequality, this is at most
\begin{equation}\label{eq:CS}
\sum_{\eta \in \text{Em}(H)} p^{-T_H} \cdot |M(H)|\cdot |\text{Em}(H)| \cdot N(H,\eta)^2.
\end{equation}

Therefore, we would like to bound $|\text{Em}(H)|$.
In this section, we give bounds on $|Em(H)|$ for any (not necessarily good) CFPG $H$.

Any embedding $\eta \in \text{Em}(H)$ induces an embedding of every connected component $C \in C_W$ into the plane. In addition, $\eta$ tells us how to fit these embedded components together - which component goes in which face. In fact, these two pieces of information determine $\eta$ completely (up to an isotopy), and we can use this to get a bound on $|\text{Em}(H)|$.

\begin{lem}\label{lem:counting_embeddings}
Let $H$ be a CFPG, and let $C_W$ be the set of its white connected components. Then there is a constant $c_3>0$ such that
\[
|\text{Em}(H)| \leq \left(\prod_{C\in C_W}{|\text{Em}(C)|}\right) \cdot c_3^m m^{|C_W|}.
\]
\end{lem}
\begin{proof}
We show that given embeddings of the connected components, there are at most $ c_3^m m^{|C_W|}$ ways to fit the connected components together, where $c_3>0$ is a constant.

Given embeddings of the connected components in $C_W$, we describe a one to one correspondence between embeddings $\eta \in \text{Em}(H)$ and edge-labelled rooted trees on $|C_W|+1$ vertices. Given $\eta$, we construct such a tree as follows. We choose a black face $b$ in some canonical way (for example, as the minimal face in some total ordering of cyclical sequences of vertices). This face will correspond to the root of the tree, and all of the other vertices in the tree will correspond to connected components in $C_W$. By stretching $b$ to be the outer face, we can think of $\eta$ as a planar map. We connect $b$ to all the components that it surrounds. Recursively, whenever a component $C$ contains an inner black face that surrounds a component $C'$, we connect $C$ and $C'$ by an edge, and label that edge according to the black face that they share. Therefore, the number of possible labels that the edges connecting a component $C$ to its children can receive is the number $B_C$ of black faces of $C$.

In the other direction, given such a labeled tree, we can construct $\eta$ recursively as follows. Note that the outer black face $b$ is already determined by the embeddings of the connected components. We associate $b$ with the root of the tree, and put all of the components corresponding to $b$'s children in the tree inside $b$. For each component $C$, we put the components corresponding to its children inside inner black faces of $C$. The black face we choose for such a child $C'$ is specified by the label of the edge $\{C,C'\}$.

Now, a tree $T$ has at most $\prod_{C \in C_W}{B_C^{d_T(C)-1}}$ possible labelings, where $d_T(C)$ is the degree of the vertex corresponding to $C$ in $T$.
We can count these labeled trees using a weighted version of Cayley's formula for trees \cite{Cayley}.
\begin{theorem}
Associate a variable $x_v$ to every vertex $v \in [n]$, and associate the monomial $\prod_{v \in [n]}x_v^{d_T(v)}$ to every $n$-vertex tree $T$. There holds
\[
\sum_T \prod_{v \in [n]}x_v^{d_T(v)} = \prod x_i\left(\sum_{v \in [n]}{x_v}\right)^{n-2}.
\]
\end{theorem}
Thus, we are interested in the sum
\[
\sum_{T \text{: tree on } C_W \cup \{b\}}\prod_{C \in C_W}{B_C^{d_T(C)-1}}  \leq
\prod B_C\left(\sum{B_C}\right)^{|C_W|-2} \leq
c_3^m m^{|C_W|}.
\]
The last inequality is a consequence of $\sum B_C = O(m),|C_W|=O(m)$ (see Proposition \ref{prop:planar_estimates}), and $\prod B_C\leq \left(\frac{\sum B_C}{|C_W|}\right)^{|C_W|}\leq e^{O(m)}$ by the inequality of means and the fact that the function $f(x)=\left(\frac{M}{x}\right)^x$ is maximized by $x=M/e$.

\end{proof}

\begin{nn}
For a CFPG $H$ with $m$ vertices, let
\[em(H):=\left(\prod_{C\in C_W(H)}{|\text{Em}(C)|}\right) \cdot m^{|C_W(H)|}.\]
\end{nn}

We now estimate $|\text{Em(C)}|$ for a given white component $C$ with $m_0$ vertices.
Our strategy is to first simplify the graph corresponding to $C$, without changing the number of vertices and black and white faces, decreasing $|\text{Em(C)}|$, or creating new extra bananas. We will show that we may, without loss of generality, assume that $C$ has the following properties.

\begin{enumerate}
\item
There is at most one vertex $x$ whose removal disconnects $C$ into more than two parts. For any other vertex $u$ whose removal disconnects $C$, the removal of its single leaf which does not contain $x$ results in a new banana, one of whose non-special vertices is $u$.
\item
There is a pair of vertices, denoted by $y,z$ with a maximal number of branches. All other pairs of vertices have at most $3$ branches.
\item
If $y,z$ have at least $3$ branches then for any other pair of vertices, $u,v$ there is a single branch containing $y,z.$ The other possible branches must either be at most one branch containing the edge $uv$, and at most one banana. If there are no such branches between $u,v$, they may have another branch, exactly if $u,v$ belong to the same branch of $y,z$, and if removing each one of $u,v$ separately from that branch, disconnects it.
\end{enumerate}
We call such a graph \emph{bipolar}.

\begin{figure}
 \centering
 \includegraphics[width=130
 mm,height=55mm]{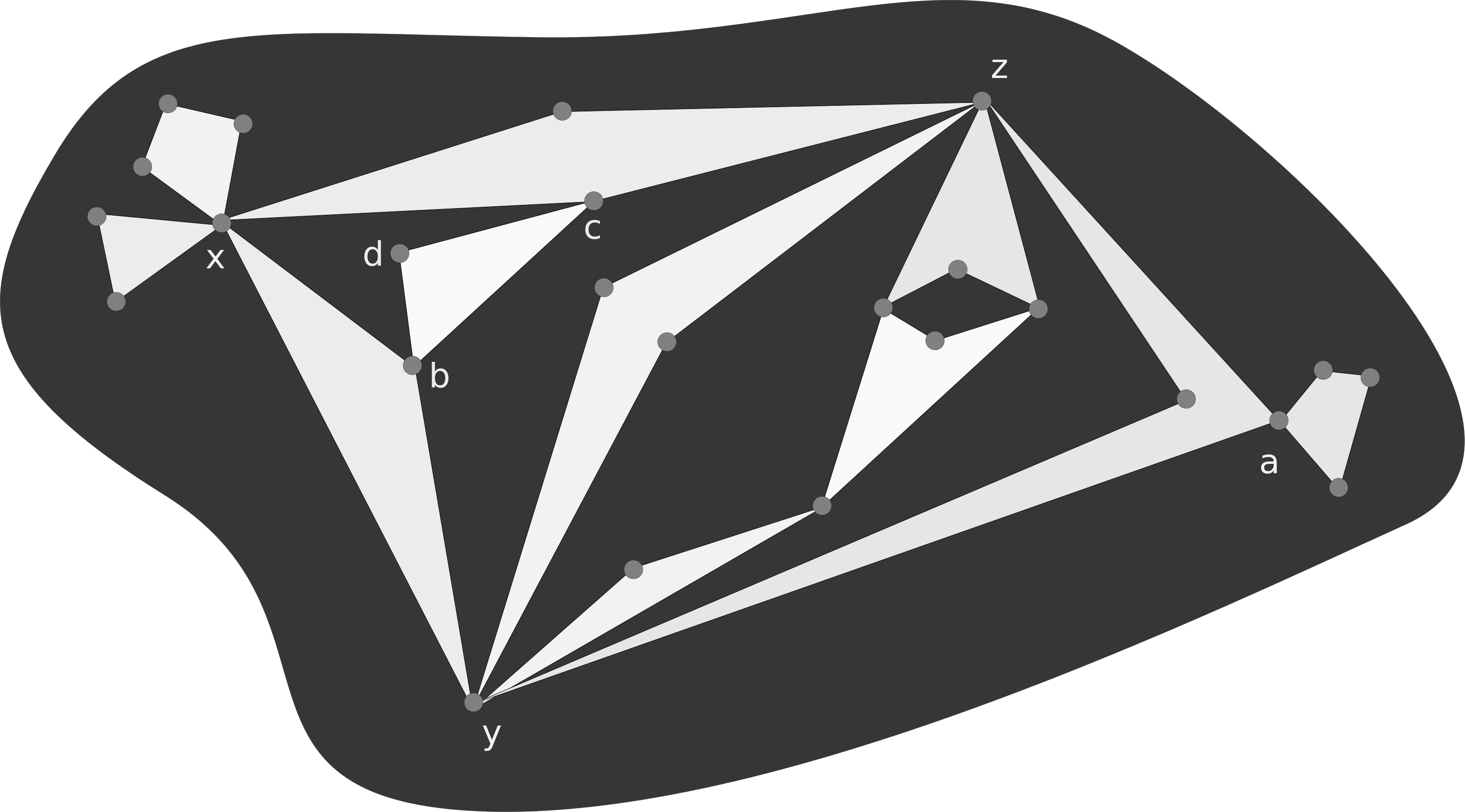}
 \caption{An example of a bipolar graph. Note that the removal of $a$'s leaf would create a new extra banana between $y$ and $z$, and that $bcd$ cannot be moved, as it contains the edge $bc$.}
\end{figure}

Let $x$ be a vertex of maximal degree in $C$, and recall that a leaf of a vertex $y$ is a connected component of the trimming $C_y$.
For a vertex $u \neq x,$ let $l_u$ be the leaf of $u$ that contains $x$.

We would like to modify $C$ by moving all of $u$'s leaves other than $l_u$ to be leaves of $x$. However, this may create a new extra banana, in which $u$ is one of the non-special vertices. In this case, we choose the minimal leaf $l'_u$ of $u$ not containing $x$ according to some total ordering of the leaves of $C$, and move all of $u$'s leaves other than $l_u, l'_u$ to be leaves of $x$. We call this a \emph{leaf step} from $u$ to $x$, and denote the resulting component by $\tilde{C}$. The following lemma states that the number of embeddings can only grow as the result of such steps.

\begin{lem}
$|\text{Em}(C)| \leq |\text{Em}(\tilde{C})|$.
\end{lem}

\begin{proof}

The vertices of a white face $f$ of $C$ have two cyclic orderings. For a given planar embedding $\eta$ of $C$, we define the orientation of $f$ induced by $\eta$ to be the counter-clockwise cyclic ordering of $f$'s vertices.
Thus, an embedding $\eta$ induces a vector $\pi(\eta)$ of orientations for all of $C$'s white faces.

Let
\[
\text{Em}_{\phi}(C):=\{\eta:\pi(\eta)=\phi\},
\]
and note that $\text{Em}(C)$ is the disjoint union of $\text{Em}_{\phi}(C)$. Therefore, it is sufficient to prove that $|\text{Em}_{\phi}(C)| \leq |\text{Em}_{\phi}(\tilde{C})|$ for all orientation vectors $\phi$.

Fix an orientation $\phi$ of the white faces of $C$. This orientation induces a directed graph structure on $C$, by orienting all of the edges so that they have a white face to their left. Hence, all of the embeddings in $\text{Em}_{\phi}(C)$ induce the same directed graph structure. Note that in any such embedding of $C$, the edges at any vertex $v$ considered in counterclockwise order alternate between incoming and outgoing, and in particular the out-degree at $v$ is equal to the in-degree.

The proof idea is to construct an injective mapping from $\text{Em}_{\phi}(C)$ to $\text{Em}_{\phi}(\tilde{C})$.
For this purpose, we will need some definitions. Let $l,l'$ be two leaves at a vertex $v$. For a given embedding of $C$, we say that a leaf $l$ is \textit{nested} in $l'$ if $l$ is bounded by an inner black face of $l'$.

Consider a leaf step from a vertex $u \neq x$ to $x$, and assume first that erasing $u$'s leaves does not create a new banana, one of whose non-special vertices is $u$.

Recall that $l_u$ is the unique leaf at $u$ that contains $x$. For the moment, we forget about the other leafs of $u$, and construct an injection from $u$'s local black faces in $l_u$ to the local black faces of $x$. This injection will tell us where to embed $u$'s other leaves in the embedding of $\tilde{C}$ that we are constructing.

Since $x$ is a vertex of maximal degree, there exists an injection $h'_u$ from the outgoing edges of $u$ in $l_u$ to the outgoing edges of $x$.
Now, in any embedding of $C$, we define an injection $h_u$ from the local black faces of $l_u$ at $u$, to the local black faces of $x$.
Every local black face $b$ of $u$ in $l_u$ has a unique outgoing edge $e_b$ at $u$, and every outgoing edge $e$ of $x$ touches a unique local black face $f_e$ of $x$. We define $h_u(b) = f_{h'(e_b)}$.


We claim that moving all of $u$'s leaves to be leaves of $x$ can only increase the number of embeddings. Indeed, if we fix an embedding of $u$'s leaves, in order to define an embedding of $C$ we must specify the nesting structure of these leaves at $u$, and, for leaves which are bounded by the same black face, their relative order. Given any such embedding of $C$, one can define an embedding of $\tilde{C}$ by moving all of $u$'s leaves bounded by a black face $b$ of $l_u$ to the local black face $h(b)$ of $x$, without changing their order and nesting structure. If under the embedding there are already leaves of $x$ nested in $h(b)$, put the leaves coming from $u$ to their left.

Note that this operation does not change the orientation induced by the embedding. The fact that $h$ is injective implies that this map is injective, and so we have not decreased the number of embeddings. Note that $x$'s degree remains maximal after such a step.

If erasing $u$'s leaves creates a new banana, we proceed as before, except that we define $h'_u$ to be an injection from the outgoing edges of $u$ in both $l_u$ and $l'_u$ to the outgoing edges of $x$. As before, this enables us to construct an injective mapping from embeddings of $C$ to embeddings of $\tilde{C}$.

\end{proof}

Suppose now that $C$ has several pairs of points with more than one branch between them. If one of these pairs contains $x$, set $z=x$ and set $y$ to be the vertex that maximizes the number of branches between $x,y$. Otherwise, let $y,z$ be the pair with the maximal number of branches.

Let $u,v$ be a different pair of vertices with at least two branches. We can move a branches between $u,v$ to be a branch of $y,z$, by replacing $u$ by $y$ and $v$ by $z$ in all of the branches edges and faces. A \emph{branch step} from $u,v$ to $y,z$ is defined by moving all the branches between $u,v$ to $y,z$, with four exceptions:
\begin{itemize}
\item We do not move a branch if it contains the edge $uv$, as this might create a double edge at $yz$.
\item We do not move the minimal banana between them, as this could increase the number of extra bananas.
\item We clearly cannot move a branch if it contains $y$ or $z$. Note that if $y,z$ have at least three branches, then there must be a single branch containing both $y$ and $z$.
\item We do not move a branch if its removal would disconnect a branch from $y$ to $z$.
\end{itemize}

Note that these exceptions can account for at most three branches of $u,v.$
Let $C'$ denote the resulting CFPG. We claim that this change can only increase the number of embeddings, except for one case in which the number of embeddings can decrease by a factor of $m_0$.

Assume first that $x \notin \{u,v,y,z\}$ and that $u,v,y,z$ are all different. Assume that $y,z$ have $a$ branches, and that $u,v$ have $b$ branches, of which $b'$ are not moved.

One can extend an embedding $\eta$ of the branches between $u,v$ to an embedding of $C$ by ordering the branches of $u,v$. In this way, we obtain $b!$ embeddings of $C$ that extend $\eta$.

On the other hand, one can extend $\eta$ to an embedding of $C'$ by ordering the remaining $b'$ branches of $u,v$, and then choosing a single ordering for the branches of $y,z$ together with the branches of $u,v$ that we moved, obtaining $b'! \cdot (a+b-b')!/a!$ embeddings.

Since $b! \leq b'! \cdot (a+b-b')!/a!$ when $a \geq b'$, this implies that we have not decreased the number of embeddings. Summing over all embeddings $\eta$ shows that $|Em(C)| \leq |Em(C')|.$

When $x\in\{u,v\},$ which may happen if new branches were created during the process, we shall move the branches of $y,z$ to $u,x=v$ and redefine $z$ to be $x$ and $y$ to be $v.$ In this case, the same analysis reveals that if originally there were two branches between $x,u$ and if there are three branches of $y,z$ we cannot move because of the restrictions above, in this case, the number of embeddings may decreases by at most a multiplicative factor of $ 3(b+1)!/2b! \leq m_0$, but this event may happen only once.


This argument extends with minor changes to the cases $x \in \{y,z,u,v\}$ and the case that $u,v,y,z$ are not all different.

To summarize, we have proved the following lemma.
\begin{lem}
In a branch stpe, $|\text{Em}(C)| \leq |\text{Em}(C')|,$ unless it is the step in which we redefine $z=x.$ In the latter case $|\text{Em}(C)| \leq m_0 |\text{Em}(C')|$.
\end{lem}

We now perform all of the possible leaf steps, followed by all of the possible branch steps, and if new leaves or branches are created during the process, we continue with leaf and branch steps until no such step is possible. Note that we will redefine $z$ to be $x$ in at most one of the steps. This process is finite, as every step but one increases the number of embeddings. It follows from the definitions of branch steps and leaf steps that the process ends with a bipolar graph.

Thus, in order to maximize $|Em(C)|$, under our constraints, we may assume that $C$ is bipolar. Such components have the following property.

\begin{prop}\label{prop:Whitney}
There is a constant $c_4$ such that the number of embeddings of a bipolar component $C$ is bounded, up to an exponential factor $c_4^{m_0}$,  by the number of ways to order the branches between $z,y$ and choose the ordering and nesting structure of the leaves of $x$.
\end{prop}
The proof is in the appendix.

We now bound $Em(C)$ in terms of the number of leaves at $x$ and branches between $y,z$.
Suppose that there are $b$ branches between $y$ and $z$, and that there are $a$ leaves at $x,~A_1,\ldots,A_a,$ in addition to the special leaf $A_0$ containing $y,z.$ Suppose $A_i$ has $a_i$ local black faces (including the external one, which is defined only after embedding in the plane).

\begin{lem}\label{lem:Ran}
The number of embeddings of $C,$ times $m_0,$ is
\[m_0c_4^{m_0} |Em(A_0)| \cdot \frac{(2a-1+\sum a_i)!}{(a+\sum a_i)!} \cdot \prod_{i=0}^a a_i\leq b! \cdot c_5^{m_0} m_0^{a-1}. \]
for some constant $c_5.$
\end{lem}
\begin{proof}
Fix an embedding of $A_0$.
We first determine the number of ways to order the branches between $z,y$ and embed the leaves, and then we add the exponential factor that should be added by Proposition \ref{prop:Whitney}. Similarly to Lemma \ref{lem:counting_embeddings}, the nesting of the leaves defines a tree structure. However, here there are $c!$ ways to nest $c$ leaves in a given local black face.

Imitating the argument of Lemma \ref{lem:counting_embeddings} we get, for a nesting tree $T$, a product over its vertices $v$, which correspond to the leaves $A_i$, of the term
\[\sum_{\{\alpha_j\}_{j=1}^{a_i},~\sum \alpha_j=d_T(v)}\binom{d_T(v)}{\alpha_1,\ldots,\alpha_{a_i}}\prod_{j=1}^{a_i}\alpha_j!=\frac{(d_T(v)+a_i-1)!}{(a_i-1)!}.\]
The sum of these products over all possible trees was calculated in \cite{TesCayley} and was proven to equal to the expression in the statement of the lemma. The right hand side of the statement of the lemma follows from the facts that $a\leq m_0/3$ and $\sum a_i\leq m_0/2$ and Proposition \ref{prop:Whitney}.

\end{proof}

The following bounds will be useful later.

\begin{lem}\label{lem:embed_comp}
Let $m_0 \leq m$. If $C_0$ is a connected CFPG with $m_0$ vertices, $B$ black faces, $W$ white faces and $\lozenge_0$ extra white bananas, then
\[
(m_0+W-B)/2 - \log_m(|Em(C_0)|) \geq (m_0-2\lozenge_0)/18-m_0\log_m c_5,
\]
and if $100 < m_0 \leq m/2$, then
\[
(m_0+W-B)/2 - \log_m(|Em(C_0)|) \geq 2+(m_0-2\lozenge_0)/18-m_0\log_m c_5,
\]
where $c_5$ is the constant from Lemma \ref{lem:Ran}.

\end{lem}

\begin{proof}
By the above discussion, up to $c_5^{m_0}$, $|Em(C_0)|$ is at most $|\tilde{Em}(C)|$, which we define as the number of ways to embed the leaves and branches of a bipolar graph $C$ that can be obtained from $C_0$. We therefore prove the analogous claims for $C,$ without the $m_0\log_m c_5$ term. Let $\lozenge\leq\lozenge_0$ be the number of extra bananas of $C.$
Let $\alpha$ be the number of leaves at $x$, and let $\beta$ number of branches of $y,z$. By Lemma \ref{lem:Ran}
\[\log_m(|\tilde{Em}(C)|)\leq \log_m(m_0)(\alpha +\beta-1).\]
Thus,
\begin{equation}\label{eq:shitbird}
(m_0+W-B)/2-\log_m(|\tilde{Em}(C)|)\geq (m_0+W-B)/2-\alpha-\beta+1+(1-\log_m(m_0))\beta.
\end{equation}
Now, starting from the empty graph, add the branches between $y,z$ to the graph one by one, starting with the branch containing the edge $yz$, if there is such a branch. We keep track of the contribution of each branch that we add to the right hand side of \eqref{eq:shitbird}.

Note that the value corresponding to the first branch $K_0$ is at least
\[
\frac{m(K_0)+W(K_0)-B(K_0)-1}{2}+(1-\log_m(m_0)) \geq \frac{m(K_0)}{12} + (1-\log_m(m_0)),
\]
by Lemma \ref{lem:W-B}. Note that every branch that we add splits a black face in two.

Any other branch $K$ contributes at least
\[\frac{m(K)-2+W(K)-B(K)-1}{2}-1 + (1-\log_m(m_0)) \] to the right hand side of \eqref{eq:shitbird}.
This is $1-\log_m(m_0)$ for a banana, at least $\frac{1}{2}$ for a component with $5$ vertices that does not contain the edge $yz$, and at least $\lceil \frac{m(K)-5}{2}\rceil/2\geq m(K)/12$ in general, by Lemma \ref{lem:W-B}.
A similar analysis for leaves reveals that when we add a leaf $K$, the contribution to \eqref{eq:shitbird} is at least  \[ (m(K)-1+W(K)-B(K))/2 - 1\geq \lceil \frac{m(K)-1}{2}\rceil/2\geq m(K)/12.\]
Therefore, we have
\[(m_0+W-B)/2-\log_m(|\tilde{Em}(C)|)\geq (m_0-2\lozenge)/12+ (1-\log_m m_0)\lozenge.\] This is clearly at least $(m_0-2\lozenge_0)/18$.

Now, the inequality in the second item holds unless $\frac{m_0-2\lozenge}{18}+2> \frac{m_0-2\lozenge}{12}$, implying that $72>m_0-2\lozenge$, and $\left( 1-\log_m(m_0) \right) \lozenge <2$.
A simple analysis shows that these two constraints cannot hold when $100<m_0<m/2$.

\end{proof}

\subsection{Step three: Analysis in terms of embedded graphs}\label{subsec3}
Our goal is to understand
\begin{equation} \label{eqn:old_sum}
\frac{1}{|S_n|^2}\sum_{H~\text{is good}}\sum_{\eta \in \text{Em}(H)} p^{-T_H} \cdot |M(H)|\cdot |\text{Em}(H)| \cdot N(H,\eta)^2.
\end{equation}

We now rearrange the terms in this sum, and introduce new notations.
Recall that a CFPG consists of a planar graph $G$, together with the structure of its white faces and an allocation of interior points for the white faces. Therefore, the sum over all good CFPGs can be expressed as a sextuple sum: We are summing over the number of vertices $m$, subsets $S\subseteq [n]$ of size $m$, colored planar graphs $H$ on the vertex set $S$ in which the white faces are specified, the number $k$ of interior points in these white faces, subsets $R \subseteq [n]\setminus S$ of size $k$, and allocations of the points in $R$ for the white faces.

Consider the set of embedded spanning planar graphs on $m$ vertices, together with a 2-coloring of their faces. Let $GP_m$ be the set of such graphs that are good in the sense of Definition \ref{def:good}. Let $M(G,k)$ be the number of triangulations of the black faces of $G$, using $k$ interior vertices in the black faces. For $G \in GP_m$, let $\bar{G}$ be the same embedded planar graph, with the colors reversed: black faces become white, and white faces are black. Note that $M(\bar{G},k)$ is the sum over all allocations of $k$ interior points for the white faces of $G$ of $|M(H)|$, where $H$ is the CFPG determined by $G$ and the allocation.

Let $T_{G,k}$ be the number of white triangles in a triangulation of the white faces using $k$ interior vertices in the white faces. For $G \in GP_m$, we define $em(G):=em(H)$ where $H$ is the underlying planar graph, together with the structure of its white faces. Recall by Lemma \ref{lem:counting_embeddings} that $em(H)\geq |Em(H)|$.

In Subsection \ref{subsec1}, we defined $C_W(F)$ for completable $F$'s and $C_W(H)$ for a CFPG $H$. We now define $C_W(G)$ for embedded colored $G \in GP_m$.
By choosing an arbitrary face of $G$ to be the outer face, we may think of $G$ as being embedded in the plane. We now define $C_W(G)$ to be the set of connected components of $G$, considered as a graph, that are surrounded by a black face, and similarly define $C_B(G)$ as the set of connected components of $G$ that are surrounded by a white face. Note that the size of $C_W$ in this new definition is equal to the size of $C_W(H)$ for the CFPG $H$ that corresponds to $G$, and that $C_B$ and $C_W$ do not depend on the choice of the outer face.

It suffices to prove that
\begin{equation}\label{eq:want to prove}
\frac{1}{|S_n|^2}\sum_{m=3}^{n}C^m\binom{n}{m}\sum_{G \in GP_m} \sum_{k=0}^{n-m}\binom{n-m}{k} p^{-T_{G,k}} M(\bar{G},k) M(G,n-m-k)^2 em(G)= o(1)
\end{equation}
for any constant $C>0.$

We will modify \eqref{eq:want to prove}, by changing the sum over all graphs $G \in GP_m$ to a maximum over all such graphs. Denote the summand corresponding to $G$ by $h(G)$, and note that it does not depend on the labelings of the vertices in $G$. Let $GP^u_m$ denote the set of \textit{unlabeled} good embedded planar graphs on $m$ vertices. By the orbit-stabilizer Theorem,
\[
\sum_{G \in GP_m}h(G) = \sum_{G' \in GP_m^u} |\{G \in GP_m: G \text{ is a labeling of } G'\}|\cdot h(G') =
\]
\[
\sum_{G' \in GP_m^u}{\frac{ h(G')\cdot m!}{|Aut(G')|}} \leq |GP_m^u|\max_{G' \in GP_m^u}{\frac{ h(G')\cdot m!}{|Aut(G')|}} \leq c_1^m \max_{G \in GP_m}{\frac{ h(G)\cdot m!}{|Aut(G)|}}.
\]
The last inequality follows from Lemma \ref{lem:num_of_planar_graphs}.

Therefore,
\begin{align}\label{eq:max}
\frac{1}{|S_n|^2}&\binom{n}{m}\sum_{G \in GP_m} \sum_{k=0}^{n-m}\binom{n-m}{k} p^{-T_{G,k}} M(\bar{G},k) M(G,n-m-k)^2 em(G)\leq\\
\notag&\leq \frac{c_1^m}{|S_n|^2}\frac{n!}{(n-m)!}\max_{G\in GP_m}\left\{\frac{em(G)}{|Aut(G)|}\sum_{k=0}^{n-m}\binom{n-m}{k} p^{-T_{G,k}} M(\bar{G},k) M(G,n-m-k)^2\right\}.
\end{align}

We now turn to bound the expressions $M(G,k)$.
If there are $b$ black faces, then we need to partition the interior vertices into $b$ sets, one for each black face, and then we can use the bounds on $T_{m_1, ... ,m_l,k}$ proved in Lemma \ref{lem:estimation_of_triangulations_of_sphere_with_holes}.

For a face $f$, let $l_f$ be the number of its boundary components, that is, the number of connected components of $G$ it touches. Let $m_1^f, ... ,m_{l_f}^f$ be the number of edges in $f$'s boundary components.

We have
\begin{equation}\label{eq:triangulation_of_F}
M(G,k) \leq \sum_{k_1+...+k_b = k}\binom{k}{k_1, ... ,k_b}\prod_{f \in F_b}T_{k_f,m_1^f, ... ,m_{l_f}^f}.
\end{equation}
The reason for the inequality is that the $m_i^f$ edges of the $i$-th component may be supported on less than $m_i^f$ vertices. Thus, although each triangulation of the face $f$ can be lifted to a triangulation of a sphere with boundary components of sizes $m_i^f,$ the opposite may not be true, because the identification of vertices in a boundary component may create loops and multiple edges.

Write $m_f=\sum m_i^f.$ By Proposition \ref{prop:planar_estimates}, the sum of $m_f$ over all faces $f$ of one color is less than $3m$, the bound for number of edges. The same proposition shows that $b$, the number of faces of one color, is at most $m$.

Equation \eqref{eq:triangulation_of_F}, together with Lemma \ref{lem:estimation_of_triangulations_of_sphere_with_holes} implies that
\begin{equation}\label{eq:simplyfying_stuff}
M(G,k)\leq k!\gamma^k c_2^{m}\sum_{k_1+...+k_b = k}\prod_{f \in F_b}(k_f+m_f)^{l_f-\frac{7}{2}},
\end{equation}


By \eqref{eq:simplyfying_stuff}, the right hand side of \eqref{eq:max} is bounded by
\begin{align*}\label{blabla}
\frac{n!(c_2^{3} c_1)^m\gamma^{2n-2m}}{|S_n|^2}\max_{G\in GP_m} \frac{em(G)}{|Aut(G)|}&\sum_{k=0}^{n-m} \gamma^{-k} p^{-T_{G,k}} \left(\sum_{k_1+...+k_b = k}\prod_{f \in F_w}(k_f+m_f)^{l_f-\frac{7}{2}}\right)\cdot\\
&\cdot(n-m-k)!\left(\sum_{k_1+...+k_b = n-m-k}\prod_{f \in F_b}(k_f+m_f)^{l_f-\frac{7}{2}}\right)^2
\end{align*}

\begin{nn}
We say that a connected component $C$ of $G$ is \emph{lonely} if it does not separate any two connected components of $G$.
For a face $f$ of the graph $G$ and a connected embedded planar graph $T$, write $r^0_T(f)$ for the number of boundary components of $f$ which belong to a lonely connected component of $G$ isomorphic to $T.$ Write $r_T(f)$ for the number of boundary components of $f$ which belong to a connected component of $G$ isomorphic to $T,$ and $r^{\neq 0}_T(f) = r_T(f)-r^0_T(f).$
Let $T_i$ be the isomorphism type of a cycle of length $i$, let $T_{5,2}$ be the isomorphism type of two nested triangles that share a vertex, and let $T_{5,1}$ be the isomorphism type of two non-nested triangles that share a single vertex. Finally, let $T_{3,1}$ be the isomorphism type of two nested vertex-disjoint triangles.
\end{nn}

Note that renaming the isomorphic lonely components inside any face $f$ results in a graph that is isomorphic to $G$. We define $R(f):= \prod_{T\in \{T_3,T_4,T_{5,1}\}} r^0_T(f)!$ if $f$ is a white face, and $R(f):= \prod_{T\in \{T_3,T_4,T_{5,2},T_{3,1}\}} r^0_T(f)!$ if $f$ is a black face.
We have
\[\prod_{f \in F_w}R(f)\prod_{f \in F_b}R(f)\leq |Aut(G)|.\]

Write
\begin{align} \label{def_Amk}
A_{m,k}(G) := 
\gamma^{-k}p^{-T_{G,k}} \left(\sum_{\sum_{f \in F_w}k_f = k}\prod_{f \in F_w}\frac{(k_f+m_f)^{l_f-\frac{7}{2}}}{R(f)}\right)\cdot \\
\notag \quad(n-m-k)!\left(\sum_{\sum_{f \in F_b}k_f = n-m-k}\prod_{f \in F_b}\frac{(k_f+m_f)^{l_f-\frac{7}{2}}}{\sqrt{R(f)}}\right)^2
\end{align}
\begin{dfn}
We say that a function $g(n,m)$ is $o_m(1)$ if for any constant $c$, the sum $\sum_{m\geq 3} c^m g(n,m))$ tends to zero when $n$ tends to infinity.
\end{dfn}
Applying Theorem \ref{thm:Tutte} to estimate $|S_n|$, it will be enough to prove that $(\frac{n^7}{n!}\sum_{k=0}^{n-m}A_{m,k}(G)) \cdot em(G)$ is bounded uniformly on $GP_m$ by some function $g(n,m)=o_m(1).$

\subsection{Step four: Reduction to flat graphs}\label{subsec4}
We now define an operation on planar graphs which simplifies the graph, at the cost of increasing $A_{m,k}$ by a controlled quantity.
\begin{dfn}
Components of type $T_3,T_4$ or $T_{5,1}$ that are contained in a white face, and components of type $T_3,T_4,T_{5,2}$, as well as pairs of triangles of type $T_{3,1}$ contained in a black face, are called \emph{exceptional components}.
Consider $G\in P_m,$ and let $f_1,f_2$ be two faces of the same color, such that $f_1$ has at least two boundary components.
By choosing $f_1$ to be the outer face, we may think of $G$ as a planar graph.

\begin{figure}
 \centering
 \includegraphics[width=130
 mm,height=55mm]{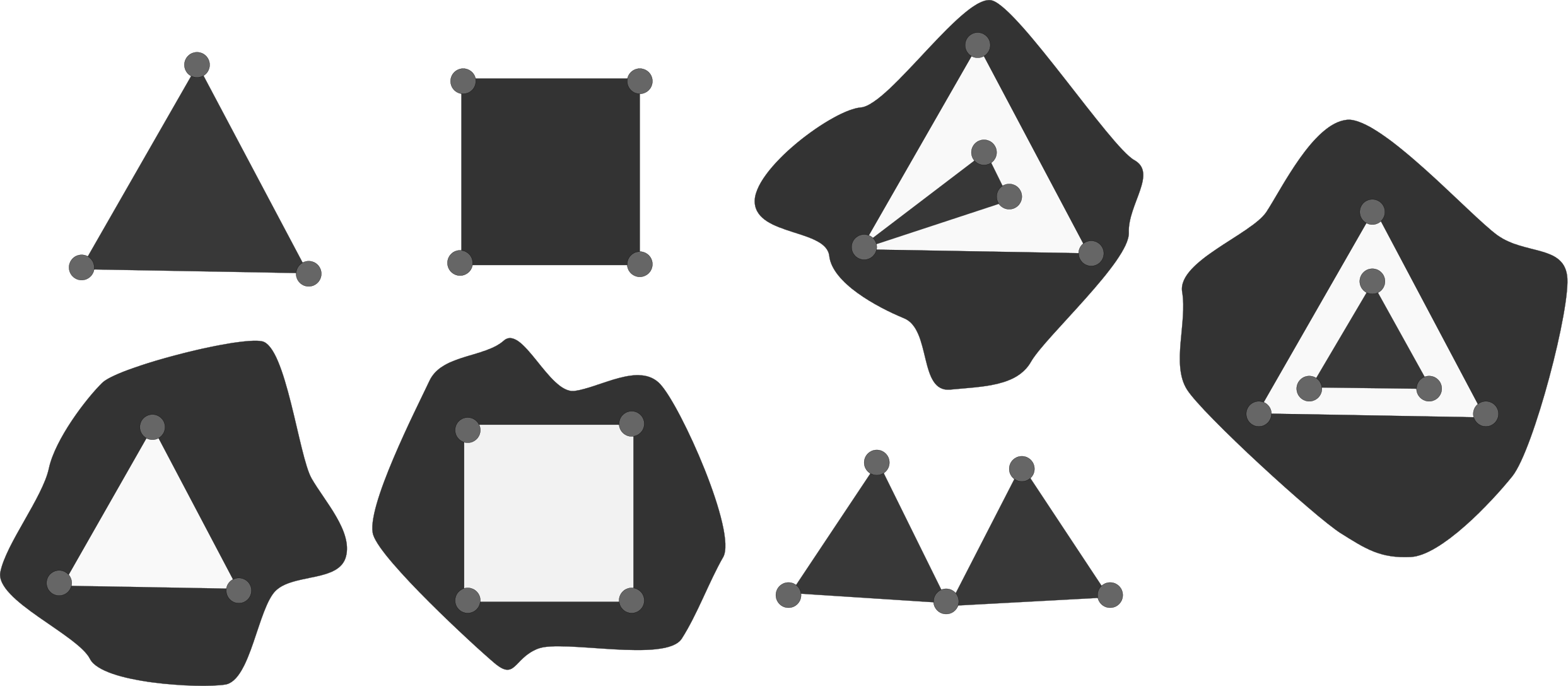}
 \caption{These are the 7 types of exceptional components. They are the components that contribute to the automorphism terms, and that receive special treatment in the flattening process}
\end{figure}

If $f_2$ does not belong to an exceptional component, we define a \emph{flattening} from $f_2$ to $f_1$ as the operation of taking all of the connected components of $G$ surrounded by $f_2$, and moving them, together with the parts of $G$ in the planar region they bound, to the interior of $f_1$. If $f_2$ is exceptional, all of the connected components of $G$ surrounded by $f_2$, except for one component, are moved to $f_1$, together with the parts of $G$ in the planar region they bound. If $f_2$ surrounds several components, not all of whom are triangles, then one of the non-triangular components remains in $f_2$. In particular, if $f_2$ is a white triangle that surrounds a black triangle as well as a non-triangular component, then we leave a non-triangular component in $f_2$.

If $f_2$ is a white triangle and also a connected component of size three that surrounds more than one component, all of whom are triangles, then we view $f_2$ together with one of the triangles it surrounds as a component of type $T_{3,1}$. In this case, we move all but two of the triangles $f_2$ surrounds to $f_1$.


We say that a graph $G \in P_m$ is \emph{flat} if $G$ has a white face $f_w$ and a black face $f_b$ such that
\begin{itemize}
\item The faces of $G$ surround at most two components, except possibly $f_w$ and $f_b$.
\item If $f \notin \{f_b,f_w\}$ is a face of $G$ that surrounds two components, then $f$ is a white triangle surrounding two black triangles.
\item The remaining faces that surround a component belong to exceptional components.
\end{itemize}
\end{dfn}
%
%

\begin{lem}\label{lem:effect of flattening}
For any constant $\delta>0$ and $G \in GP_m$, there exists a flat graph $G'$ that can be obtained from $G$ by a sequence of flattenings, which satisfies
\[
A_{m,k}(G')\geq \tilde\CC_{\delta}^{-m} \cdot (1-\delta)^k \cdot A_{m,k}(G),
\]
for some constant $\tilde\CC_\delta>1$ that depends only on $\delta$.
%
%
\end{lem}
The proof appears in the appendix.

\begin{prop}\label{prop:good_flat}
A flat graph obtained by the flattening process applied to a good graph satisfies
\begin{equation}\label{eq:good_flat}m-2\lozenge \geq m/100\end{equation}
\end{prop}
\begin{proof}
In a good graph $\frac{24m}{49}\geq\lozenge.$ The flattening process may add at most $C_B$ extra bananas, since any new extra banana is the result of taking some component counted in $C_B$ from a square face, thus \[|C_B|\geq \lozenge-\frac{24m}{49}.\]
On the other hand \[m-2\lozenge \geq 3|C_B|,\] because removing the extra bananas does not change connectivity, and each of the disjoint components counted by $C_B$ must have at least three vertices.
Thus, \[m-2\lozenge\geq 3(\lozenge-24m/49)\Leftrightarrow 121m/49\geq 5\lozenge\Leftrightarrow m-2\lozenge\geq 3m/245,\]
and therefore
\[m-2\lozenge\geq m/100.\]
\end{proof}

Theorem \ref{thm:main} would therefore follow if we could prove that for any flat graph $G$ satisfying \eqref{eq:good_flat} there holds \[\frac{n^7}{n!}\sum_{k=0}^{n-m} A_{m,k}(G) \cdot em(G)\leq g(n,m), \text{ for a function }g(n,m)=o_m(1).\]

\subsection{Step five: Analysis for flat graphs}\label{subsec5}

We start with two simple analytic results, whose proofs are in the appendix.
The following propositions will be useful for understanding $A_{m,k}$.

\begin{prop}\label{prop:no_division}
Suppose $m_i,l_i,~i=1,\ldots,w$ are positive integers. Assume that $m_i\geq 3$ for all $i$, that $l_i\leq 3$ for $i\geq 2$, and that $l_i=3$ for $a$ values of $i$.
The following inequalities hold for
\[
L(m,l)=\sum_{\{(k_i)_{i=1}^w:~k_i\geq 0,~\sum k_i=k\}}\prod_{i=1}^w(k_i+m_i)^{l_i-\frac{7}{2}}.
\]
\begin{enumerate}
\item For all positive $\delta$ there exists a constant $C_\delta$ such that
\[
L(m,l)\leq C_\delta^{\sum m_i}(1+\delta)^k(k+m_1)^{l_1-\frac{7}{2}}.
\]
\item There exists a constant $c_6$ such that
\[
L(m,l)\leq c_6^{\sum m_i}(k+m_1)^{l_1-\frac{7}{2}+\frac{a}{2}} .
\]
\end{enumerate}

\end{prop}

\begin{prop}\label{prop:estimation_for_1-eps_terms}
Fix $0<\varepsilon<1.$ Then there exists a constant $C_\varepsilon$ such that for any integers $m,l,x$ with $-10\leq l\leq m,~3\leq x\leq m$ and any positive $k,$
\[(1-\varepsilon)^k(k+x)^{l} \leq (1-\varepsilon)^{\frac{k}{2}}C_\varepsilon^m m^l.\]
\end{prop}

Let $G$ be a flat graph satisfying \eqref{eq:good_flat}. We call the white face which touches the maximal number of connected components of $G$ the \emph{distinguished white face}. If there is no unique maximal white face, choose any maximal white face $f$ with minimal $m_f$ to be the distinguished one. We denote the distinguished white face by $f_w$, and similarly define the distinguished black face $f_b$. By choosing the distinguished black face $f_b$ to be the outer face, we may assume that $G$ is planar. For any face $f$, any connected component of $G$ which touches it and does not separate it from the distinguished face will be called a \emph{hole}. We also say that the corresponding connected component of $G$ is \emph{surrounded} by $f$. A \emph{white (black) hole} is a hole surrounded by a white (black) face.

Let $\delta=\varepsilon/10$. Note that in a flat graph, the only black face $f$ with $l_f \geq 3$ is the distinguished face $f_b$. Therefore, Proposition \ref{prop:no_division} allows us to bound $A_{m,k}(G)$ for a flat graph $G$ by
\begin{equation}\label{eq:endless}
c_3^m(1+\delta)^{2k}\gamma^{-k} p^{-T_{G,k}} \frac{(k+m_w)^{l_w-\frac{7}{2}}}{R(f_w)}(n-m-k)!\frac{(n-k-\frac{m}{2})^{2l_b(G)-7}}{R(f_b)},
\end{equation}
where $c_3$ is a new constant, and $l_w=l_{f_w},m_w=m_{f_w},l_b=l_{f_b},m_b=m_{f_b}$.
Here we used Item (a) in order to bound the term corresponding to the white faces, and Item (b) to bound the term corresponding to the black ones. We can write $(n-k-\frac{m}{2})$ instead of $n-k-m+m_b$ because the ratio $(n-k-m+m_b)/(n-k-m/2)$ is bounded by a constant.

Note that as by Proposition \ref{prop:num_of_triangles}, $T_{G,k} \geq 2k$, we have
\[
(1+\delta)^{2k} p^{-T_{G,k}} = (1+\delta)^{2k} (1-\varepsilon)^{T_{G,k}} \left(\frac{e}{\gamma n}\right)^{-T_{G,k}} \leq (1-\varepsilon)^k \left(\frac{e}{\gamma n}\right)^{-T_{G,k}}.
\]

Therefore, by plugging in the value of $p$ and applying Stirling's approximation, we can upper bound \eqref{eq:endless}
\begin{align}\label{eq:111}
c_3^mn! \left(\frac{e}{\gamma n}\right)^{k+m                                                                                                                                                                                                                                                                                                                                                                                                                                                                                                                                                                                                                                                                                                                                                                                                  -T_{G,k}/2} (1-\varepsilon)^{k} \frac{(k+m_1)^{l_w-\frac{7}{2}}}{R(f_w)}\left(\frac{n-m-k}{n}\right)^{n-m-k}\frac{\left(n-k-\frac{m}{2}\right)^{2l_b-7}}{R(f_b)}.
\end{align}

In order to estimate $k+m-T_{G,k}/2,$ we need the following proposition:

\begin{prop}
\label{prop:k+m-F/2}
Let $G \in P_m$ with $r$ connected components. There holds:
\[
m+k-T_{G,k}/2 = \frac{1}{2}(m+r+W-B+3)- |C_B| . 
\]
\end{prop}

\begin{proof}
By Proposition \ref{prop:num_of_triangles}, we have
\begin{align*}
T_{G,k} & = \sum_{f\in F_w} {\left(2 k_f + 2 l_f - 4 + \sum_{i=1}^{l_f}{m^f_i}\right)}\\
& = 2k + E - 2W + 2 \sum_{f\in F_w}{(l_f - 1)}.
\end{align*}
By Euler's formula for planar graphs, we have $F-E+V = r+1$, where $F=B+W$ is the total number of faces and $V = m$ is the number of vertices. This implies that $ E/2 = \frac12 (B+W+m-r-1)$, and so we have:
\begin{align*}
k+m-T_{G,k}/2 & = k+m - k - E/2 + W - \sum_{f\in F_w}{(l_f - 1)} \\
 & = W + m - \frac{1}{2}(B+W+m-r-1) - \sum_{f\in F_w}{(l_f - 1)} \\
 & = \frac{1}{2}(m+r+W-B+3)-|C_B|.
\end{align*}
\end{proof}


In particular, $(k-T_{G,k}/2)$ is bounded by a constant times $m$, and does not depend on $k$.
Therefore, by Proposition \ref{prop:estimation_for_1-eps_terms} we can bound expression \eqref{eq:111} by
\begin{equation}\label{eq:when will it end}
n! c_\varepsilon^m \frac{ (1-\varepsilon)^{k}}{n^{m+k-T_{G,k}/2}}\left(\frac{n-m-k}{n}\right)^{n-m-k}  \frac{m^{l_w-\frac{7}{2}}}{R(f_w)}\frac{\left(n-k-\frac{m}{2}\right)^{2l_b-7}}{R(f_b)},
\end{equation}
where $\CC_\varepsilon$ depends only on $\varepsilon.$

Using Proposition \ref{prop:estimation_for_1-eps_terms} again, we have

\[
\sum_{k=0}^{n-m}\frac{(1-\varepsilon)^{k}}{n^{m+k-T_{G,k}/2}}\left(\frac{n-m-k}{n}\right)^{n-m-k}  \frac{m^{l_w-\frac{7}{2}}}{R(f_w)}\frac{\left(n-k-\frac{m}{2}\right)^{2l_b-7}}{R(f_b)} \leq
\]
\begin{align*}
\sum_{k=0}^{\min\{n-m,2\varepsilon n\}}\frac{(1-\varepsilon)^{k}}{n^{m+k-T_{G,k}/2}}\left(\frac{n-m-k}{n}\right)^{n-m-k}  \frac{m^{l_w-\frac{7}{2}}}{R(f_w)}\frac{\left(n-k-\frac{m}{2}\right)^{2l_b-7}}{R(f_b)}+\\
\qquad+\CC_{\varepsilon}^m\sum_{k>\min\{n-m,2\varepsilon n\}}^{n-m} \frac{(1-\varepsilon)^{k}}{n^{m+k-T_{G,k}/2}}(1-\varepsilon)^{\frac{n-k-m}{2}} \frac{m^{l_w-\frac{7}{2}}}{R(f_w)}\frac{m^{2l_b-7}}{R(f_b)}
\end{align*}


Now,
\[
\sum_{k=0}^{\min\{n-m,2\varepsilon n\}}\left(\frac{n-m-k}{n}\right)^{n-m-k}\leq
\sum_{k=0}^{2\varepsilon n}e^{-(m+k)(n-m-k)/n}=O(1),
\]
Thus the contribution of the sum over the first $\min\{n-m,2\varepsilon n\}$ terms is upper bounded by
\[
n^{-\left(\frac{m+r+W-B+3}{2}-|C_B|\right)}\frac{m^{l_w-\frac{7}{2}}}{R(f_w)} \cdot \frac{\left(\alpha_\varepsilon n\right)^{2l_b-7}}{R(f_b)},
\]
for a constant $\alpha_\varepsilon$ which depends only on $\varepsilon.$

For $n-m\geq 2\varepsilon n$, there holds
\[\sum_{k>\min\{n-m,2\varepsilon n\}}^{n-m}(1-\varepsilon)^{\frac{n-m-k}{2}+k} \cdot m^{2l_b-7}\leq
n(1-\varepsilon)^{\varepsilon n}m^{2l_b-7}.\]
For any fixed $\varepsilon<1,~n$ large enough, $n(1-\varepsilon)^{\varepsilon n}<1$. Hence, for all $n$ larger than some $N=N_\varepsilon$ we have
\[n(1-\varepsilon)^{\varepsilon n}m^{2l_b-7}\leq m^{2l_b-7} \leq n^{2l_b - 7}.\]

To summarize, we can upper bound $\frac{n^7}{n!}\sum_{k=0}^{n-m}A_{m,k}(G)$ by
\[n^7c_\varepsilon^m\sum_{k=0}^{n-m}\frac{ (1-\varepsilon)^{k}}{n^{m+k-T_{G,k}/2}}\left(\frac{n-m-k}{n}\right)^{n-m-k}  \frac{m^{l_w-\frac{7}{2}}}{R(f_w)}\frac{\left(n-k-\frac{m}{2}\right)^{2l_b-7}}{R(f_b)}\leq\]

\[\leq (\CC^1_\varepsilon)^m\left(\frac{1}{n}\right)^{\frac{m+r+W-B+3}{2}-|C_B|-2l_b}\frac{m^{l_w-\frac{7}{2}}}{R(f_w)R(f_b)}=\]
\[= (\CC^1_\varepsilon)^m\left(\frac{m}{n}\right)^{d'}\frac{m^{-d''}}{R(f_w)R(f_b)}\]
where $\CC^1_\varepsilon$ is a constant, and $d',d''$ are defined as follows.

\[d'(G) := \frac{m+r+W-B+3}{2}-|C_B|-2l_b,\]
\[d''(G) := 5+\frac{m+r+W-B}{2}-|C_B|-l_w-2l_b,\]

Let
\[B_m(G)=
\left(\frac{m}{n}\right)^{d'}\frac{m^{-d''}}{R(f_w)R(f_b)}.\]

\subsection{Last step: Analysis of lonely components}\label{subsec6}
It remains to show that $B_m(G) \cdot em(G)$ is bounded uniformly on the set of flat graphs obtained from good planar graphs with $m$ vertices, by $g(n,m)=o_m(1).$ This will follow from a similar estimation for $B_m(G) \cdot em'(G)$,
 where
\[em'(G) := c_5^{-m}m^{|C_W|}\prod_{C\in C_W \cup C_B,~\text{of size at least 1000}}|Em(C)|,
\]
where $c_5$ is the constant of Lemma \ref{lem:Ran}.
Indeed, for smaller components we can bound the number of embeddings by a constant to the power of the size of the component. The division by $c_5^m$ does not affect our attempt, by the definition of $o_m(1)$.

Let $d(G) = d''(G)-\log_m em'(G)$, and recall that we are considering a flat graph satisfying \eqref{eq:good_flat}.
We would like to understand $d$ and $d'$. We will see that $d' \geq m/6$, and ideally we would like to show that $d$ is at least a constant times $m$. When this does not hold, we will see that the automorphism terms $R(f_w)$ and $R(f_b)$ compensate.

Starting with the empty graph, we add the components of $G$ one by one, and keep track of the contributions to $d$ and $d'$ after each such step. For a component $C$ of $G$, let $m(C)$ and
$\lozenge(C)$ be the number of its vertices and extra white bananas respectively. We will show that the contribution of a component $C$ to $d$ is at least $\frac{m(C)-2\lozenge(C)}{20}$, if it is either not exceptional, or an exceptional component that is not contained in one of the distinguished faces $f_w,f_b$.

The exceptional components contained in the distinguished faces may have a negative or zero contribution to $d$, and so they must be dealt with separately. In the case of lonely exceptional components, we can make up for their negative contributions by using the automorphism terms. On the other hand, any non-lonely exceptional component $C$ that is contained in a distinguished face must surround an additional component $C'$. We will show that the combined contribution of $C$ and $C'$ is at least $(m(C)+m(C')-2\lozenge(C'))/20$.

Let $G_1,\ldots, G_r$ be the connected components of $G$, ordered so that $G_1$ touches the distinguished white face $f_w$, and for every $i$ the component $G_i$ does not separate $G_1$ and $G_j$ with $j<i$.
For our purposes, two nested triangles contained in $f_b$ are treated as a single component. If a triangle $t$ in $f_b$ contains two triangles, then we choose one arbitrarily, and together with $t$ it forms a single component.

 We have
\[
d = 5+\frac{m-r+W-B}{2}-l_w-2l_b - \sum_{i:m(G_i) \geq 1000}{\log_m(|Em(G_i)|)}.
\]

Let $G_{\leq i}$ be the embedded graph formed by the connected components $G_1,\ldots, G_i$, whose faces, are colored in the unique way for which the face $f_w$ is colored white.
We now define $\Delta_i$ and $\Delta_i'$ to be the contributions of the component $G_i$ to $d$ and $d'$ respectively. That is, $\Delta_1 = d(G_{\leq 1}), \Delta_1'=d'(G_{\leq 1})$, $\Delta_i = d(G_{\leq i}) - d(G_{\leq i-1})$, and $\Delta_i' = d'(G_{\leq i}) - d'(G_{\leq i-1})$.

In order to understand $\Delta_i$ for different types of components $G_i$,
we first partition the components of $G$ into six different sets. We define
\begin{align*}
& R=\{i:G_i \text{ is a lonely triangle contained in $f_b$ or $f_w$}\} \\
& S=\{i:G_i \text{ is a lonely exceptional component that is not a triangle contained in $f_b$ or $f_w$}\} \\
& T=\{i:G_i \text{ is a non-lonely exceptional component contained in $f_b$ or $f_w$}\}
\end{align*}

We define an injection $\pi:T \rightarrow \{1, ... ,r\}\setminus (S \cup R \cup T)$ by choosing for each $j \in T$ an index $j$ such that the component $G_j$ is surrounded by a face of $G_i$ in $G$. Finally, let $X$ denote the indices of components $G_i$ such that $m(G) \geq 1000$, and let $Y$ be the set of all remaining indices, namely $\{1,...,r\}\setminus (R\cup S \cup T \cup \pi(T) \cup X)$.

\begin{prop} \label{prop:bounds_for_one_component}
\begin{enumerate}
\item If $i \in X$, then $\Delta_i \geq \frac{m(G_i)-2\lozenge(G_i)}{20},$ except for at most one $i$, for which $\frac{m(G_i)-2\lozenge(G_i)}{20}>\Delta_i \geq \frac{m(G_i)-2\lozenge(G_i)}{20}-2$. \label{item:first}

\item If $i \in R$, then $\Delta_i \geq -\frac{1}{2}$. \label{item:second}

\item If $i \in S$, then $\Delta_i \geq 0$. \label{item:third}

\item If $i \in Y$, then $\Delta_i \geq \frac{m(G_i)}{20}$. \label{item:fourth}

\item If $i \in T$, then $\Delta_i+\Delta_{\pi(i)} \geq \frac{m(G_i)+m(G_{\pi(i)})-2\lozenge(G_{\pi(i)})}{20}$, except for at most one $i$, for which
\[
\frac{m(G_i)+m(G_{\pi(i)})-2\lozenge(G_{\pi(i)})}{20}>\Delta_i+\Delta_{\pi(i)} \geq \frac{m(G_i)+ m(G_{\pi(i)})-2\lozenge(G_{\pi(i)})}{20}-2.
\] \label{item:fifth}

\end{enumerate}

\end{prop}

\begin{proof}
Item \ref{item:first} is a consequence of Lemma \ref{lem:embed_comp}. There is at most one component that does not satisfy the stronger inequality, since the size of each such component, by Lemma \ref{lem:embed_comp} again, must be greater than $m/2$.

Consider some component $G_i$ with $m(G_i)<1000$. Lemma \ref{lem:W-B}, which bounds the difference between the number of black faces and the number of white faces for a single component in terms of the number of its vertices, implies that
\begin{itemize}
\item If $G_i$ belongs to a distinguished face, then $\Delta_i\geq \lceil\frac{m(G_i)-5}{2}\rceil/2$.
\item If $G_i$ belongs to a non-distinguished white face, then $\Delta_i\geq \lceil\frac{m(G_i)-1}{2}\rceil/2$.
\item If $G_i$ belongs to a non-distinguished black face, then $\Delta_i\geq \lceil\frac{m(G_i)+3}{2}\rceil/2$.
\end{itemize}
The case of two nested triangles contained in $f_b$ is a special case, which should be checked separately. Recall that the contribution to $r$ is $2$, and so the contribution to $d$ is $(6-2+1-1)/2-2 = 0$ as claimed.
Items \ref{item:second} and \ref{item:third} follow.

Moreover, one can verify directly that the contribution of components in $Y$ with $5$ vertices is at least $\frac{1}{2}$. For example, a component of type $T_{5,1}$ surrounded by a non-distinguished white face has $m(G_i)=5, W=0,B=2$ and does not contribute to $l_b,l_w$, and therefore its contribution is $(5-1-2)/2=1$. As the items above imply that for components of size 6 or more $\Delta_i \geq m(G_i)/12$, item \ref{item:fourth} follows.

For item \ref{item:fifth}, the hope is that the nonpositive contribution of $G_i$ will be compensated by the positive contribution of $G_{\pi(i)}$. We consider two cases: If $m(G_{\pi_i})<1000$, then note that as $G_{\pi(i)}$ does not lie in a distinguished face, its contribution is at least $\lceil\frac{m(G_{\pi_i})-1}{2}\rceil/2$, which is always greater than $(m(G_{\pi_i})+m(G_i))/20$, as $m(G_i) \leq 6$ and $m(G_{\pi(i)})\geq 3$.

If $m(G_{\pi_i})\geq 1000$, then by Lemma \ref{lem:embed_comp}, for all but at most one $i$ there holds
$\Delta_{\pi(i)} \geq \frac{m(G_{\pi(i)})-2\lozenge(G_{\pi(i)})}{18}$, and since $m(G_i)\leq 6$ this implies that $\Delta_i+\Delta_{\pi(i)} \geq \frac{m(G_i)+m(G_{\pi(i)})-2\lozenge(G_{\pi(i)})}{20}$, as desired.
\end{proof}

As $d =  \sum_{i=1}^r{\Delta_i}$, Proposition \ref{prop:bounds_for_one_component} allows us to lower bound $d$, up to an additive term of $O(1)$, by
\begin{align}\label{eq:bound_for_d}
-\frac{1}{2}|S| + \sum_{i \in T}{\frac{m(G_i) + m(G_{\pi(i)})-2\lozenge(G_{\pi(i)})}{20}} + \sum_{i \in X \setminus \pi(T)}{\frac{m(G_i)-2\lozenge(G_i)}{20}} + \sum_{i \in Y}{\frac{m(G_i)}{20}}.
\end{align}

A similar analysis on $d'$ is much simpler to perform. Indeed, Lemma \ref{lem:W-B} implies that every component contributes at least $\frac{m(G_i)-1}{4}$, and in particular, $d'\geq m/6$. Here it is important to note that there is no additive $O(1)$ term.

Thus, \[B_m(G)em'(G)
\leq \left(\frac{m}{n}\right)^{m/6}\frac{m^{-d}}{{R(f_w)R(f_b)}}.\]

We will now show that if $m$ is greater than some constant $M$,
\begin{equation}\label{eq:final}
\frac{m^{-d}}{{R(f_w)R(f_b)}}\leq A^m m^{-\frac{m}{5000}}.
\end{equation}
for some constant $A.$ This will complete the proof, as $(\frac{m}{n})^{m/6}m^{-m/5000}=o_m(1),$ while for $m$ smaller than $M$, the naive estimation of $(\frac{m}{n})^{m/6}m^m$ is enough.

If $d\geq m/5000,$ \eqref{eq:final} follows and we are done.

Suppose now that $d<m/5000$. We will show that in this case, there are many small empty components in the distinguished faces. This will enable us to lower bound $R(f_w)R(f_b)$, and to compensate for the small values of $d$.

Let $Q$ denote the number of empty exceptional components that lie in distinguished faces.
That is, $Q = |R|+|S|$.
Observe that
\begin{align*}
m -2\lozenge(G)= &\sum_{i=1}^r m(G_i) - 2 \lozenge(G_i) \leq 13Q +20d + O(1).
\end{align*}
This may be verified by comparing the coefficients of the components in each of the six sets.

On the other hand, by \ref{prop:good_flat} \[m -2\lozenge(G)\geq m/100\]

Thus,   $m/100\leq 13Q +20d + O(1).$

If $m$ is large enough, the assumption that $d<m/5000$ implies that $Q>m/5000$ and that $d+Q>m/5000$.

Recall that
\[
R(f_w)R(f_b):= (r^{0}_{T_3}(f_b))! \cdot (r^{0}_{T_3}(f_w))! \cdot (r^{0}_{T_4}(f_b))! \cdot (r^{0}_{T_4}(f_w))!\cdot(r^{0}_{T_{5,2}}(f_b))! \cdot (r^{0}_{T_{5,1}}(f_w))!  (r^{0}_{T_{3,1}}(f_w))!.
\]

Therefore, $R(f_w)R(f_b) \geq ((Q/7)!)^7 \geq(\frac{m}{35000e})^{Q}.$

Thus,
\[\frac{m^{-d}}{R(f_w)R(f_b)}\leq a^m m^{-(d+Q)}\leq a^m m^{-m/5000},\] for $m$ larger than some constant $a$. The theorem follows.

\newpage
\appendix
\section{Proofs of Technical Claims}
\begin{obs}\label{obs:useful_binom_est}
For any $\delta>0,$ there exists $C_\delta>0$ such that for any natural $r,m$
\begin{equation}\label{eq:useful_binom_est}
\binom{r+m}{m}\leq C_\delta^m(1+\delta)^r.
\end{equation}
\end{obs}
\begin{proof}
We use the bound $\binom{r+m}{m} \leq \left(\frac{(r+m)e}{m}\right)^m$. Since the limit of $((1+n)e)^{1/n}$ when $n\rightarrow \infty$ is $1$, for every $\delta>0$ there is an $M_{\delta}>0$ such that if $n\geq M_{\delta}$ then $((1+n)e)^{1/n} < 1+ \delta$. Thus, if $\frac{r}{m}\geq M_{\delta}$ then $\left(\frac{(r+m)e}{m}\right)^m \leq (1+\delta)^r$.

Conversely, if $\frac{r}{m} < M_{\delta}$, then if we take $C_{\delta}=(1+M_\delta)e$, we have $ \left(\frac{(r+m)e}{m}\right)^m \leq (C_\delta)^m$.

\end{proof}

\begin{proof}[Proof of Proposition \ref{prop:moving_all_points_of_bananas}]
We want to show that
\[\sum_{\{k_i\}_{i=1}^r: \sum k_i=k}\binom{k}{k_1,\ldots,k_r}\prod_{i=1}^r T_{k_i,4}\leq 8^{r-1} T_{k,4},\]
or equivalently,
\[\sum_{\{k_i\}_{i=1}^r: \sum k_i=k}\prod_{i=1}^r T'_{k_i,4}\leq 8^{r-1} T'_{k,4},\]
where $T'_{k,4}=T_{k,4}/k!.$
It will be enough to show the claim for $r=2$ and then use induction.
We will give a proof from the book.

The claim for $r=2$ can be verified using a computer up to $k=200,$ using the formula for $T_{k,4}$ given in Theorem \ref{thm:Tutte}.
For $k>200$, by using the same formula, and Stirling's approximation, one can show that $T'_{k,4}$ is asymptotically equal to (but always larger than)
$Z\gamma^kk^{-\frac{5}{2}}\exp\left(-\frac{\frac{7}{2}+\frac{13}{144}}{k}\right),$
where $Z=5120/243\sqrt{6\pi}\approx 4.853.$ Moreover, for $k\geq 10$
\[1<\frac{T'_{k,4}}{Z\gamma^kk^{-\frac{5}{2}}\exp\left(-\frac{\frac{7}{2}+\frac{13}{144}}{k}\right)}< 1.05.\]

Thus, \begin{align*}
\sum_{a=21}^{k/2}\frac{T'_{a,4}T'_{k-a,4}}{T'_{k,4}}&<1.11Z\sum_{a=21}^{k/2}\frac{a^{-\frac{5}{2}}(k-a)^{-\frac{5}{2}}}{k^{-\frac{5}{2}}}\\
&<1.11Z\sum_{a=21}^{k/2}\frac{a^{-\frac{5}{2}}(k/2)^{-\frac{5}{2}}}{k^{-\frac{5}{2}}}<\\
&<1.11\cdot2^{\frac{5}{2}}\left(\int_{20}^\infty x^{-\frac{5}{2}}dx\right)Z=\\
&=1.11\cdot2^{\frac{5}{2}}\cdot\frac{2}{3}\cdot20^{-\frac{3}{2}}Z<0.25.
\end{align*}
In addition
\[\sum_{a=1}^{20}\frac{T'_{a,4}T'_{k-a,4}}{T'_{k,4}}<1.05\sum_{a=1}^{20}\frac{T'_{a,4}}{\gamma^a}\left(\frac{k}{k-a}\right)^{\frac{5}{2}}.\]
For $k>200, 1<a\leq 20,~(\frac{k}{k-a})^{\frac{5}{2}}<1.31,$ and also $\sum_{a=1}^{20}\frac{T'_{a,4}}{\gamma^a}<1.25,$ as a direct verification reveals. The contribution of $a$ from $1$ to $20$ is thus no more than $1.7.$

Adding the contribution of $a=0,$ which is $2$, and multiplying everything by $2$ to take into account the summation over $a>k/2$,
we get a bound of $7.9$, as claimed.

For $r>2$, the result follows from a simple inductive argument. Indeed,
\[
\sum_{\{k_i\}_{i=1}^r: \sum k_i=k}\prod_{i=1}^{r} T'_{k_i,4} = \sum_{k_n} \left(\sum_{\{k_i\}_{i=1}^{r-1}: \sum k_i=k-k_n} \prod_{i=1}^{r-1} T'_{k_i,4} \right) \cdot T'_{k_n,4} \leq
\]
\[
\sum_{k_n} 8^{r-2} T'_{k-k_n,4} T'_{k_n,4} \leq 8^{r-1} T_{k,4}.
\]
\end{proof}

\begin{proof}[Proof of Proposition \ref{prop:Whitney}]
Recall Whitney's famous theorem (\cite{Whitney})
\begin{theorem}
A $3-$connected planar graph has a unique embedding in the sphere, up to a reflection.
\end{theorem}
We would like to show that bipolar CFPGs, although not necessarily $3$-connected, have a controlled number of embeddings.

We will define an injection from embeddings of $C$ to a certain set of CFPGs. We will give an upper bound on the number of CFPGs that we can get in this fashion, which will imply an upper bound on the number of embeddings of $C$.

We now define an operation that changes $C$.
Let $e_1=\{a,b\},e_2=\{a,c\}$ be two edges of $C$ that share a vertex $a$, and that bound the same black face in some embedding of $C$.  Let $w_1,w_2$ be the white faces touched by $e_1,e_2$ respectively. By adding two new vertices to each edge, we split $e_1$ and $e_2$ into three segments. That is, the edge $\{a,b\}$ is replaced by the three edges $\{a,b_2\},\{b_2,b_1\}$ and $\{b_1,b\}$, and $\{a,c\}$ is similarly split.
 Next, we remove the edges $\{b_2,b_1\}, \{c_2,c_1\}$ and add the edges $\{b_2,c_2\},\{b_1,c_1\}$, merging the faces $w_1,w_2$ into a single white face. We call this operation adding a \textit{ribbon} between $e_1$ and $e_2$.

We call a connected CPFG $3$-connected if there are no two points whose removal disconnects the CFPG \emph{as a topological space}.
Note that Whitney's theorem holds for $3$-connected CFPGs, as it can be reduced to the theorem for graphs by replacing each white face by a $3-$connected triangulation.

Consider a connected bipolar CFPG $C$. We first prove the proposition for a single leaf $L$ at $x$.
Let $m_L$ be the number of vertices spanned by $L$, and let $y,z$ be the points in $L$ between which there are several branches, assuming that there are such points. Assume first that $x \notin \{y,z\}$.

Fix an ordering of the branches, and let $\eta$ be an embedding of $L$ that induces the chosen ordering on the branches. We now define the extension of $L$ that corresponds to $\eta$. Our goal is to make $L$ $3$-connected by adding ribbons to it.

Let $B,B'$ be two consecutive branches between $y$ and $z$. Let $e_1=\{y,b\},e_2=\{y,c\}$ be two edges bounding the same black face such that $e_1$ belongs to $B$ and $e_2$ belongs to $B'$. We add a ribbon between $e_1$ and $e_2$. This is done for each pair of consecutive branches, and the same procedure is applied to $z$. Let $d_B(y)$ be the degree of $B$ at $y$. The number of ways to add a ribbon between two branches $B,B'$ is at most $d_B(y) \cdot d_{B'}(y)$, and so for a given ordering of the branches, the number of ways to perform this step is at most $\prod_{B} d_B(y)^2 d_B(z)^2$, which is at most $2^{2\sum_B{d_B(y)+d_B(z)}}=2^{O(m_L)}$.

Next, we deal with bananas and branches in $L$ between vertices $u,v$ other than $y$ and $z$ that contain the edge $\{u,v\}.$ Assume that $T$ is such a branch between two vertices $u$ and $v$, one of whom may be equal to $y$ or $z$. Without loss of generality, assume that $d_u \leq d_v$, and let $e_1,e_2$ be two edges of $u$ such that $e_1$ belongs to $T$ and $e_1,e_2$ bound the same black face under the embedding $\eta$. We add a ribbon between $e_1$ and $e_2$. The number of possible choices for $e_1,e_2$ is at most $2 d_u$. Let $G$ be the planar graph in which two vertices are connected if there is a banana or a triangular face between them in $L$. The number of ways to perform this step is at most
\[\prod_{\{u,v\} \in E(G)} \min(d_v,d_u) \leq 2^{\sum_{\{u,v\} \in E(G)} \min(d_v,d_u)},\] which is $2^{O(m_L)}$ by Lemma \ref{lem:mckay}.

Let $B$ be a branch between $y$ and $z$, and let $v$ be a vertex of $B$ whose removal would disconnect $B$.
There are several different cases that we have to deal with.

First of all, there may be an additional leaf at $v$. As $C$ is bipolar, $v$ can have at most one additional leaf. As above, we add a ribbon between an edge $e_1$ of $v$ contained in this leaf, and another edge $e_2$ of $v$ not contained in the leaf, such that $e_1,e_2$ bound the same black face under $\eta$. We do this for all such vertices $v$, and as the number of choices for $v$ is at most $d_v^2$, the number of ways to perform this step is $2^{O(m_L)}$.

Another possibility is that $v$'s removal disconnects $B$ into two parts $B_1,B_2$ such that $B_1$ contains $y$ and $B_2$ contains $z$. We add a ribbon between two of $v$'s edges $e_1,e_2$ contained in $B_1,B_2$ respectively, that bound the same black face under $\eta$. As above, the number of ways to do this is $2^{O(m_L)}$.

We claim that the resulting CFPG is $3$-connected (as a topological space). Indeed, consider the trimming at two vertices $a,b$. As a result of the ribbons, we cannot detach any leaves, triangles, branches or bananas. It may be possible to split a branch $B$ between $y$ and $z$ in the middle, but in this case there are two branches between $a$ and $b$. Therefore there are at least two branches between $y$ and $z$, and so splitting $B$ does not disconnect the CFPG.

As a result, Whitney's theorem implies that there is a unique embedding of $L$ up to a reflection. As $\eta$ induces an embedding of the CFPG, we have defined a mapping from embeddings of $L$ to $3$-connected CFPGs such that at most two embeddings are mapped to each CFPG. As we saw, given an ordering of the branches between $y$ and $z$, the number of CFPGs that we can get using this procedure is $2^{O(m_L)}$, and therefore the number of embeddings of $L$ is $2^{O(m_L)}$.

We now apply the same argument to our original bipolar component $C$. Given an embedding $\eta$ of $C$, we perform the above procedure on each of the leaves separately. Then, given the ordering and nesting structure of the leaves at $x$, we say that two leaves are consecutive if they touch the same black face. For each pair of consecutive leaves $L,L'$ at $x$, we add a ribbon between an edge $e_1$ of $x$ belonging to $L$ and edge $e_2$ of $x$ that belongs to $L'$, both bounding the same black face. Given the ordering and nesting structure of the leaves, there is a unique way to perform this step.
 This connects all of $x$'s leaves with ribbons, creating a $3$-connected CFPG. The number of possible CFPGs is $2^{O(m)}$, and therefore given the ordering and nesting structure of the leaves at $x$ and the ordering of the branches at $y,z$, the number of embeddings is $2^{O(m)}$.

\end{proof}

\begin{proof}[Proof of Lemma \ref{lem:effect of flattening}]
We shall first perform flattening steps between white faces.
If all of the white faces touch a single component, then there are no white steps. Otherwise, let $f_1$ be a white face that touches a maximal number of components of $G$. In addition, if $l_{f_1}\geq 4$, then among those faces which touch the maximal number of components of $G,~f_1$ is the one which touches the largest number of edges. Otherwise, if $l_{f_1}<4$, then we choose $f_1$ to be the white face which touches the minimal number of edges, among those faces that touch the maximal number of components of $G$.

All of the steps that we perform will be flattenings from some white face to $f_1$.
The flattening steps will not change the above assumptions. We order the flattening steps so that steps involving a triangle that contains only triangles occur last.

We note that a flattening step does not affect $T_{G,k}$, and so the only change to $A_{m,k}$ made by a flattening step from a face $f_2$ to $f_1$ is to the term
\[
\sum_{k_1+k_2 = k}\left(\frac{(k_1+m_1)^{l_1-\frac{7}{2}}}{R(f_1)}\right) \cdot \left(\frac{(k_2+m_2)^{l_2-\frac{7}{2}}}{R(f_2)}\right).
\]
Here $k$ is the number of interior points assigned to $f_1$ and $f_2$, or $k-\sum_{f \in F_w \setminus \{f_1,f_2\}}k_f$ in the notation of \ref{def_Amk}.

Assume first that $f_2$ is not exceptional.
For convenience write $a_i,b_i,c_i$ for $r_{T_3}^0(f_i)$, $r_{T_4}^0(f_i)$, $r_{T_{5,1}}^0(f_i)$ respectively, for $i=1,2.$ Set $m_i^{''} = 3r_{T_3}^0(f_i)+4r_{T_4}^0(f_i)+5r_{T_{5,1}}^0(f_i).$
Write $l'_i$ for the number of other components of $G~f_i$ touches, except the one which separates $f_1,f_2.$ For the separating component which touches $f_i,$ write $m'_i$ for the number of edges of it which touch $f_i.$ Write $m_i^0$ for the number of edges of $G$ which touch $f_i$ but are not included in the separating component or in the lonely components of types $T_3,T_4,T_{5,1}.$

We first show
\begin{align}\label{eq:main_for_flat}
&\notag\sum_{k_1+k_2=k}\prod_{i=1,2}\frac{(k_i+m'_i+m^0_i+m^{''}_i)^{l'_i+a_i+b_i+c_i-\frac{5}{2}}}{a_i!b_i!c_i!} \leq\\
& \quad\quad 100 \left(\frac{a+b+c}{a_1+b_1+c_1}\right)^3 e^{3m'_2} \frac{(k+m'_1+m^0+m^{''})^{l'+a+b+c-\frac{5}{2}}}{a!b!c!}(m_{2}')^{-\frac{5}{2}},
\end{align}

where $a=a_1+a_2, b=b_1+b_2, c=c_1+c_2$, and $(\frac{a+b+c}{a_1+b_1+c_1})$ is defined to be $1$ if the denominator vanishes.

If  $l'_1+a_1+b_1+c_1 = l'_2+a_2+b_2+c_2 = 1,$
\begin{align*}
&(k+m'_1+m^0+m^{''})^{-\frac{1}{2}}/(k+1)\geq(k+m'_1+m^0+m^{''})^{-\frac{3}{2}}\geq\\
&\geq(k+m'+m^0+m^{''})^{-\frac{3}{2}}\geq(k_1+m'_1+m^0_1+m^{''}_1)^{-\frac{3}{2}}(k_2+m'_2+m^0_2+m^{''}_2)^{-\frac{3}{2}}.
\end{align*}
Thus,
\[(k+m'_1+m^0+m^{''})^{-\frac{1}{2}}\geq\sum_{k_1+k_2=k}\prod_{i=1,2}(k_i+m'_i+m^0_i+m^{''}_i)^{-\frac{3}{2}}.\]
In addition \begin{equation}\label{eq:exp}
e^{2m'_2}\geq {m_2'}^{\frac{5}{2}},
\end{equation} and the inequality follows.

 When $l'_1+a_1+b_1+c_1 = 2,$ then $l'_2+a_2+b_2+c_2 \leq 2.$ In this case Equation \eqref{eq:main_for_flat} is a consequence of Proposition \ref{prop:no_division2} and \eqref{eq:exp}.

Consider the general case, $l'_1+a_1+b_1+c_1\geq 3$. For any $k_1,k_2\geq 0$ whose sum is $k$,
\begin{align*}
&\frac{(k+m'_1+m^0+m^{''})^{l'+a+b+c-\frac{5}{2}}}{a!b!c!} \\
& = (k+m'_1+m^0+m^{''})^{\frac{1}{2}}\frac{(k+m'_1+m^0+m^{''})^{l'+a+b+c-3}}{a!b!c!}\\
&~\geq(k+m'_1+m^0+m^{''})^{\frac{1}{2}}
\frac{\binom{l'+a+b+c-3}{l'_1+a_1+b_1+c_1-3}}{a!b!c!}\cdot\\
&\quad\cdot(k_1+m'_1+m^0_1+m^{''}_1)^{l'_1+a_1+b_1+c_1-3}
(k_2+m^0_2+m^{''}_2)^{l'_2+a_2+b_2+c_2}\\
&~\geq (k_2+1)^{\frac{5}{2}}\frac{\binom{l'+a+b+c-3}{l'_1+a_1+b_1+c_1-3}}{a!b!c!}\cdot\\
&\quad\cdot(k_1+m'_1+m^0_1+m^{''}_1)^{l'_1+a_1+b_1+c_1-\frac{5}{2}}
(k_2+m^0_2+m^{''}_2)^{l'_2+a_2+b_2+c_2-\frac{5}{2}}.
\end{align*}
If $l'_2+a_2+b_2+c_2\leq 2,$ then
\[(k_2+m^0_2+m^{''}_2)^{l'_2+a_2+b_2+c_2-\frac{5}{2}}\geq (k_2+m^0_2+m^{''}_2+m'_2)^{l'_2+a_2+b_2+c_2-\frac{5}{2}}.\]
Otherwise we have:
\begin{align*}
&(k_2+m^0_2+m^{''}_2)^{l'_2+a_2+b_2+c_2-\frac{5}{2}} =\\
&~= (k_2+m'_2+m^0_2+m^{''}_2)^{l'_2+a_2+b_2+c_2-\frac{5}{2}}
\left(\frac{k_2+m^0_2+m^{''}_2}{k_2+m'_2+m^0_2+m^{''}_2}\right)^{l'_2+a_2+b_2+c_2-\frac{5}{2}}\geq\\
&~\geq (k_2+m'_2+m^0_2+m^{''}_2)^{l'_2+a_2+b_2+c_2-\frac{5}{2}}
\left(\frac{k_2+m^0_2+m^{''}_2}{k_2+m'_2+m^0_2+m^{''}_2}\right)^{\frac{k_2+m^0_2+m^{''}_2}{3}}\geq\\
&~\geq (k_2+m'_2+m^0_2+m^{''}_2)^{l'_2+a_2+b_2+c_2-\frac{5}{2}}e^{-m'_2/3},
\end{align*}
where in the one before last passage we have used the fact that for any connected component of $G$ which touches a face $f,$ the number of edges which touch $f$ is at least $3$. Indeed, $(m_2^0 + m_2'')/3$ is a lower bound for $l_2-1$, which is greater than $l_2'+a_2+b_2+c_2-\frac{5}{2}$.

The last passage is a consequence of the inequality
\[\left(\frac{x}{x+y}\right)^{x}\geq e^{-y}\Longleftrightarrow 1+y/x\leq e^{y/x}\] which holds for positive reals.

In addition, by the monotonicity of $\binom{a+x}{b+x},\binom{a+x}{b},$
\begin{align*}
 &\binom{l'+a+b+c-3}{l'_1+a_1+b_1+c_1-3}\geq \binom{a+b+c-3}{a_1+b_1+c_1-3} \geq \\
& \geq \frac{1}{8}\left(\frac{a_1+b_1+c_1}{a+b+c}\right)^3\binom{a+b+c}{a_1+b_1+c_1}\geq \frac{1}{8}\left(\frac{a_1+b_1+c_1}{a+b+c}\right)^3\binom{a}{a_1}\binom{b}{b_1}\binom{c}{c_1},
\end{align*}
where $\frac{a_1+b_1+c_1}{a+b+c}$ is defined to be $1$ if the nominator and denominator vanish.

Combining the above, we see that for any fixed $k_2,$
\begin{align*}
\frac{e^{3 m'_2}}{8(k+1)}\left(\frac{a+b+c}{a_1+b_1+c_1}\right)^3  & \frac{(k+m'_1+m^0+m^{''})^{l'+a+b+c+d-\frac{5}{2}}}{a!b!c!}(m_{2}')^{-\frac{5}{2}}{(1+k_2)^{-\frac{5}{2}}}\geq\\
&\geq\prod_{i=1,2}\frac{(k_i+m'_i+m^0_i+m^{''}_i)^{l'_i+a_i+b_i+c_i-\frac{5}{2}}}{a_i!b_i!c_i!}.
\end{align*}
Since $\sum_{k_2\geq 0}{(1+k_2)^{-\frac{5}{2}}}\leq 2,$
\eqref{eq:main_for_flat} follows.

We now prove that \eqref{eq:main_for_flat} holds also when $f_2$ is exceptional, but not a triangle containing only triangles, under the assumption that we move one less face from $f_2$ than before, with the following small changes in notations:
$m_2'$ is now the number of edges in $f_2$ of the separating component together with those of the component which is not moved. The numbers $a,b,c$ will be the number of type $T_3,T_4,T_{5,1}$ components in $f_1$ after the flattening step, so that \[l'+a+b+c:=-1+\sum_{i=1,2}l'_i+a_i+b_i+c_i,~m^0=m_1^0+m_2^0,m''=m_1''+m_2''.\]

In particular, \[a_1+a_2\geq a,b_1+b_2\geq b,c_1+c_2\geq c.\]

In order to perform a flattening step from an exceptional component, we must have
\[l_1'+a_1+b_1+c_1\geq l_2'+a_2+b_2+c_2\geq2.\]
If both are exactly $2,$ the result follows again from Proposition \ref{prop:no_division2} and \eqref{eq:exp}.

The remaining case, $l_1'+a_1+b_1+c_1\geq 3$ follows almost as before, only that this time for a fixed $k_2$ we get
\begin{align*}
&\frac{(k+m'_1+m^0+m^{''})^{l'+a+b+c-\frac{5}{2}}}{a!b!c!} \geq (k_2+1)^{\frac{3}{2}}\frac{\binom{l'+a+b+c-3}{l'_1+a_1+b_1+c_1-3}}{a!b!c!}\cdot\\
&\quad\cdot(k_1+m'_1+m^0_1+m^{''}_1)^{l'_1+a_1+b_1+c_1-\frac{5}{2}}
(k_2+m^0_2+m^{''}_2)^{l'_2+a_2+b_2+c_2-\frac{5}{2}}.
\end{align*}
Again $\sum_{k_2\geq 0}(1+k_2)^{-\frac{3}{2}}\leq 3,$ and the remaining steps are identical.

Iterating \eqref{eq:main_for_flat} over the different possible white faces (not a triangle with only triangles inside) sequentially, the $100e^{3m'_2}$ are collected to give at most $(100e^9)^m,$ as there are at most $m$ steps, and the all the $m'_2$ are distinct and sum to at most $3m.$ The $(\frac{a_1+b_1+c_1}{a+b+c})^3$ terms give a telescopic product which is no more than $\frac{1}{m^3},$ since the final $a+b+c$ may sum to at most $m.$

We now move to the last case, that $f_2$ is a triangle with $a_2>2$ triangles inside. It thus has $3a_2+3$ edges. We assume that all other types of white faces have already been dealt with. Suppose that $f_1$ has $a_1$ triangles and $l'$ other holes, not containing the one which separates $f_1$ from $f_2,$ and $m_1$ edges not involved in its $a_1$ triangles. We have $l'+a_1\geq 3$, because the number of holes in $f_1$ is always maximal.
We would like to show that for any positive $\delta$ there exists $c_\delta$ such that for every $k_1,k_2$ that sum to $k$,
\begin{align}
\label{eq:last???}
c_\delta^{a_2}(1+\delta)^{k_2}\frac{(k_1+m_1+3a_1+3a_2-6)^{l'+a_1+a_2-4}}{(a_1+a_2-2)!}(k_2+9)^{\frac{1}{2}}\geq \\ \notag
\geq \frac{(k_1+m_1+3a_1)^{l'+a_1-\frac{5}{2}}}{a_1!}\frac{(k_2+3a_2+3)^{a_2-\frac{5}{2}}}{a_2!}.
\end{align}
Iterating this over all the exceptional faces of the last type will show, from similar reasoning to the above, that we decrease $A_{m,k}$ by at most $c_\delta^m(1+\delta)^k.$

Equation \eqref{eq:last???} follows from the following observations.
\begin{enumerate}
\item \[\sqrt{\frac{k_2+3a_2+3}{k_2+9}}< 2^{a_2}.\]
\item
\[\frac{(k_1+m_1+3(a_1+a_2-2))^{l'+a_1+a_2-\frac{9}{2}}}{(a_1+a_2-2)!}\geq \]\[\sqrt{(k_1+m_1+3(a_1+a_2-2))}\frac{\binom{l'+a_1+a_2-5}{l'+a_1-2}}{(a_1+a_2-2)!}\frac{(k_1+m_1+3a_1)^{l'+a_1-3}}{a_1!}\frac{(3a_2-6)^{a_2-2}}{a_2!}\geq \]\[
\frac{\binom{l'+a_1+a_2-5}{l'+a_1-2}}{(a_1+a_2-2)!}\frac{(k_1+m_1+3a_1)^{l'+a_1-\frac{5}{2}}}{a_1!}\frac{(3a_2-6)^{a_2-2}}{a_2!}\]
\item
\[a_1!a_2!\frac{\binom{l'+a_1+a_2-5}{l'+a_1-3}}{(a_1+a_2-2)!}=\frac{a_2!}{(a_2-2)!}\frac{a_1!}{(l'+a_1-3)!}\frac{(l'+a_1+a_2-5)!}{(a_1+a_2-2)!}.\]
By monotonicity of $(a+x)!/a!$ this fraction is at least $a_2(a_2-1)$ when $l'\geq 3,$ and even in the remaining cases it is at least $\frac{a_2(a_2-1)}{a_2^3}>e^{-2a_2}.$
\item \[(c_\delta(3a_2-6))^{a_2-2}(1+\delta)^{k_2}\geq(k_2+3a_2+9)^{a_2-2}\] for $c_\delta$ large enough.
Indeed, \[(\frac{k_2+3a_2+3}{c_\delta(3a_2-6)})^{a_2-2}< (1+\frac{k_2}{c_\delta(3a_2-6)})^{a_2-2}<e^{\frac{k_2}{3c_\delta}}<(1+\delta)^{k_2}\]
if $c_\delta$ is large enough.
\end{enumerate}
Combining the above and the claim follows. 

Regarding the black flattening moves, we need to show
\begin{align}
&100 \left(\frac{a+b+c+d}{a_1+b_1+c_1+d_1}\right)^3 e^{3m'_2} \frac{(k+m'_1+m^0+m^{''})^{l'+a+b+c+d-\frac{5}{2}}}{\sqrt{a!b!c!d!}}(m_{2}')^{-\frac{5}{2}}\geq\\
&\notag\quad\quad\geq\sum_{k_1+k_2=k}\prod_{i=1,2}\frac{(k_i+m'_i+m^0_i+m^{''}_i)^{l'_i+a_i+b_i+c_i+d_i-\frac{5}{2}}}{\sqrt{a_i!b_i!c_i!d_i!}} .
\end{align}
The proof is almost identical to the proof in the white case.
The sole change comes from the difference in the automorphism factors, $\sqrt{a!b!c!d!},$ where $a,b,c,d$ count the black exceptional, where in the white case we had $a!b!c!.$ The appearance of the four automorphism terms does not complicate the proof, and the fact they appear inside a square root only makes the numerical argument simpler.
\end{proof}

\begin{proof}[Proof of Proposition \ref{prop:no_division}]
For the first item, suppose first all $l_i < 4.$
Then the left hand side of the sum can be bounded by
\[\binom{k+w-1}{w-1},\]
where we just replaced the sum by the number of summands times the trivial bound $1$ for each summand.
By Observation \ref{obs:useful_binom_est}, for $0<\delta<1,$ we get the bound $C^{'w}_\delta(1+\delta)^k,$ and $w<\sum m_i.$

If $l_1>3,$ we can use the previous case to bound the left hand side by
\[C_{\delta}^{'m}\sum_{k_1=0}^k(k_1+m_1)^{l_1-\frac{7}{2}}(1+\delta)^{k-k_1}\leq
\]
\[
C^{'m}_{\delta}(k+m_1)^{l_1-\frac{7}{2}}\sum_{k_1=0}^k(1+\delta)^{k-k_1}<
C^m_{\delta}(k+m_1)^{l_1-\frac{7}{2}}(1+\delta)^{k},\]
defining a new constant $C_\delta=\frac{1+\delta}{\delta}C'_\delta.$

The second item follows from the stronger claim below.
\begin{prop}\label{prop:no_division2}
Suppose $m_i,l_i,~i=1,\ldots,w$ are positive integers, except possibly $l_1$ which may be half an integer. Assume that $m_i\geq 3$ for all $i$, and that if $l_1$ is half an integer, then $l_1\geq \frac{7}{2}.$
Assume that $l_1 \geq ... \geq l_w$ and that the first $a$ $l_i$'s are at least $4$, the next $b$ are exactly $3$ and the others are smaller.
Assume also that if $l_1=l_2,$ and if $l_1\geq 4,$ then $m_1\geq m_i$ for any $i$ with $l_i=l_1.$
If $l_1=l_b<4,$ then $m_1\leq m_i$ for any $i$ with $l_i=l_1.$

Then we have:
\[\sum_{\{(k_i)_{i=1}^w:~k_i\geq 0,~\sum k_i=k\}}\prod_{i=1}^w(k_i+m_i)^{l_i-\frac{7}{2}}\leq 16^{w-1}(k+\sum_{1\leq i\leq a}m_i)^{l_1-\frac{7}{2}+\sum_{1<i\leq a+b}(l_i-\frac{5}{2})}.\]
\end{prop}
Consider first the case $w=2$. We consider several different cases.
If $l_2\geq 4$, the result is directly seen to hold.

If $l_1\geq\frac{7}{2}$ and $l_2<3,$  then $(k_1+m_1)^{l_1-\frac{7}{2}}$ is maximized when $k_1=k,$ and
\[
\sum_{k_1=0}^k((k-k_1)+m_2)^{l_2-\frac{7}{2}} < 2.
\]

If $l_1\geq \frac{7}{2}$ and $l_2 = 3$, then again $(k_1+m_1)^{l_1-\frac{7}{2}} \leq (k+m_1)^{l_1-\frac{7}{2}}$, and so it remains to show that
\[
\sum_{k_2=0}^k{(k_2+m_2)^{-\frac{1}{2}}} \leq 16 (k+m_1)^{\frac{1}{2}}.
\]
This is a consequence of the fact that $m_2\geq 3$ and that $\sum_{i=0}^k{(i+3)^{-\frac{1}{2}}} \leq 2k^{\frac{1}{2}}$.

Suppose now that $l_2\leq l_1\leq 3$.
The last $\frac{k}{2}$ terms in the left hand side are bounded by
\begin{align*}
\left(\sum_{k_1=\lfloor\frac{k}{2}\rfloor+1}^k(k-k_1+m_2)^{l_2-\frac{7}{2}}\right)(k/2+m_1)^{l_1-\frac{7}{2}}\leq 2(k/2+m_1)^{l_1-\frac{7}{2}+\frac{\delta_{l_2-3}}{2}},
\end{align*}
where $\delta_{-}$ is Kronecker's delta. Here again we are using the fact that $\sum_{i=0}^k{(i+3)^{-\frac{1}{2}}} \leq 2k^{\frac{1}{2}}.$
If $l_1=l_2,$ then by assumption on $m_1$ the first $\frac{k}{2}$ terms of the left hand side are bounded by
\begin{align*}
\left(\sum_{k_1=1}^{\lfloor\frac{k}{2}\rfloor}(k_1+m_1)^{l_1-\frac{7}{2}}\right)(k/2+m_2)^{l_2-\frac{7}{2}}\leq 2(k/2+m_1)^{l_2-\frac{7}{2}+\frac{\delta_{l_1-3}}{2}},
\end{align*}
Otherwise $l_2<l_1,$ and
\begin{align*}
&\left(\sum_{k_1=1}^{\lfloor\frac{k}{2}\rfloor}(k_1+m_1)^{l_1-\frac{7}{2}}\right)(k/2+m_2)^{l_2-\frac{7}{2}}\leq \frac{k+1}{2}m_1^{l_1-\frac{7}{2}}((k+1)/2)^{l_1-\frac{9}{2}}=\\
&=((k+1)m_1/2)^{l_1-\frac{7}{2}}\leq(k/2+m_1)^{l_1-\frac{7}{2}}.
\end{align*}

The result for general $w$ follows now by a simple induction argument, since for $w\geq 3,$ the $w=2$ step gives
\[\sum_{\{(k_i)_{i=1}^w:~k_i\geq 0,~\sum k_i=k\}}\prod_{i=1}^w(k_i+m_i)^{l_i-\frac{7}{2}}\leq
16\sum_{\{(k_i)_{i=1}^{w-1}:~k_i\geq 0,~\sum k_i=k\}}(k_1+m'_1)^{l'_1-\frac{7}{2}}\prod_{i=3}^w(k_{i-1}+m_i)^{l_i-\frac{7}{2}}\]
with $l'_1=l_1 +l_2-\frac{5}{2}$ if $l_2\geq 3,$ and otherwise $l'_1=l_1,~m'_1=m_1+m_2,$ when $l_2\geq 4$ and otherwise $m'_1=m_1.$
\end{proof}

\begin{proof}[Proof of Proposition \ref{prop:estimation_for_1-eps_terms}]
Denote $k/m$ by $\alpha.$
For nonnegative $l,$ it suffices to estimate
\[(1-\varepsilon)^{\alpha m/2}(1+\alpha)^lm^l\leq((1-\varepsilon)^{\alpha/2}(1+\alpha))^m m^l.\]
Now, $(1-\varepsilon)^{\alpha/2}(1+\alpha)$ is bounded by a constant $C'_\varepsilon,$ and the whole expression is thus bounded by $(C'_\varepsilon)^m m^l.$

If $-10\leq l \leq 0,$
\[(1-\varepsilon)^{k/2}(k+x)^{l}\leq 1\leq (2^{10})^m m^{-10}\leq (2^{10})^mm^l.\]
Choose $C_\varepsilon = max\{2^{10},C'_{\varepsilon}\}.$
\end{proof}

\bibliographystyle{amsabbrvc}
\bibliography{bibli_comb}

\end{document}